\documentclass[11pt,reqno]{amsart}

\usepackage{amsmath,amsfonts,amssymb,amsthm,epsfig,epstopdf,url,array,comment,hyperref,mathrsfs, inputenc,mathtools}
\usepackage{hyperref}
\usepackage{setspace}
\usepackage{graphicx}
\usepackage{float}
\usepackage{tikz-cd}
\usepackage[cmtip,all]{xy}
\newcommand{\longsquiggly}{\xymatrix{{}\ar@{<~>}[r]&{}}}
\usepackage{tabularray}

\theoremstyle{plain}
\newtheorem{thm}{Theorem}[section]
\newtheorem{lem}[thm]{Lemma}
\newtheorem{prop}[thm]{Proposition}
\newtheorem{cor}[thm]{Corollary}

\theoremstyle{definition}
\newtheorem{definition}[thm]{Definition}
\newtheorem{con}[thm]{Conjecture}
\newtheorem{exmpl}[thm]{Example}

\newtheorem{qstn}[thm]{Question}

\newtheorem{rem}[thm]{Remark}
\newtheorem*{rems}{Remarks}

\newcommand{\ind}{\mathbf 1}

\newcommand{\R}{\mathbb{R}}
\renewcommand{\d}{\textnormal{ d}}
\newcommand{\<}{\langle}
\renewcommand{\>}{\rangle}

\newcommand{\grad}{\nabla}

\newcommand{\hidethis}[1]{}

\newcommand{\E}{\mathbb{E}}

\newcommand{\supp}{\textnormal{support}}

\newcommand{\vol}{\textnormal{Vol}}

\newcommand{\var}{\textnormal{Var}}

\newcommand{\del}{\partial}
\newcommand{\wasser}{\mathscr{W}_{2}}
\newcommand{\s}{\mathfrak{s}}
\newcommand{\cc}{\mathfrak{c}}

\newcommand{\RR}{\mathbb{R}}
\newcommand{\dd}{\mathrm{d}}
\newcommand{\DD}[2]{\operatorname{D}\left(#1\middle\Vert#2\right)}

\begin{document}

\begin{abstract}
We study the connection between the concavity properties of a measure $\nu$ and the convexity properties of the associated relative entropy $D(\cdot \Vert \nu)$ along optimal transport. As a corollary we prove a new dimensional Brunn--Minkowski inequality for centered star-shaped bodies, when the measure $\nu$ is log-concave with a p-homogeneous potential (such as the Gaussian measure). Our method allows us to go beyond the usual convexity assumption on the sets that is fundamentally essential for the standard differential-geometric technique in this area. 

We then take a finer look at the convexity properties of the Gaussian relative entropy, which yields new  functional inequalities. First we obtain curvature and dimensional reinforcements to Otto--Villani's HWI inequality in Gauss space, when restricted to even strongly log-concave measures. As corollaries, we obtain improved versions of Gross' Logarithmic Sobolev inequality and Talagrand's transportation cost inequality in this setting. 
\end{abstract}

\title[New Brunn--Minkowski and functional inequalities]{New Brunn--Minkowski and functional inequalities via convexity of entropy}
\author{Gautam Aishwarya}
\address{Faculty of Mathematics, Technion - Israel Institute of Technology, Haifa 3200003, Israel.}
\email{gautama@campus.technion.ac.il}
\curraddr{Department of Mathematics, Michigan State University, East Lansing 48824, USA.}
\email{aishwary@msu.edu}
\author{Liran Rotem}
\address{Faculty of Mathematics, Technion - Israel Institute of Technology, Haifa 3200003, Israel.}
\email{lrotem@technion.ac.il}
\date{}
\maketitle

\section{Introduction}

\subsection{Improved Brunn--Minkowski inequalities}
Let $\nu$ be a Borel measure on $\RR^{n}$. We say that $\nu$ satisfies
a Brunn--Minkowski inequality with exponent $a$ if for
all Borel sets $K_{0}, K_{1} \subseteq\RR^{n}$ satisfying $\nu(K_{0}),\nu(K_{1})>0$ and
every $0\le t \le1$ we have
\begin{equation} \label{eq:BM}
\nu\left((1-t)K_{0}+ t K_{1}\right)\ge\left((1-t)\nu(K_{0})^{a}+ t \nu(K_{1})^{a}\right)^{1/a}.
\end{equation}
Here $(1-t)K_{0} +t K_{1}=\left\{ (1-t) x+ t  y:\ x\in K_{0},\ y\in K_{1}\right\} $ is the $t$-\emph{Minkowski average} of the sets $K_{0}$ and $K_{1}$. Note that the inequality becomes stronger as $a$ increases. The prototypical
example is the classical Brunn\textendash Minkowski inequality, which
states that the Lebesgue measure satisfies (\eqref{eq:BM}) with exponent
$a=\frac{1}{n}$ (see e.g. \cite{Gardner02} for a survey on the
Brunn\textendash Minkowski inequality and its importance). When $\nu$ is not the Lebesgue measure, inequalities such as \eqref{eq:BM} are sometimes called weighted Brunn--Minkowski inequalities.

More generally, from the works of Borell (\cite{Borell74}, \cite{Borell75}) and Brascamp--Lieb (\cite{BrascampLieb76BBL}), a characterization of all measures $\nu$ which satisfy (\eqref{eq:BM}) for some exponent
$a$ is known: 
\begin{thm} \label{thm: Borell}
Let $\nu$ be a Borel measure on $\RR^{n}$ with density $\phi$ with
respect to the Lebesgue measure on $\R^n$, and fix $a\in[-\infty,\frac{1}{n}]$.
Then the following are equivalent: 
\begin{enumerate}
\item $\nu$ satisfies a Brunn\textendash Minkowski inequality with exponent
$a$ (we also say that $\nu$ is $a$-concave). 
\item For all $x,y\in\RR^{n}$ with $\phi(x),\phi(y)>0$ and every $0\le t \le1$
we have 
\begin{equation}
\phi\left(\left(1-t\right)x+t y\right)\ge\left((1-t)\phi(x)^{b}+t\phi(y)^{b}\right)^{1/b},
\end{equation}
 where $\frac{1}{a}=\frac{1}{b}+n$ (we also say that $\phi$ is $b$-concave). 
\end{enumerate}
\end{thm}

For simplicity we assumed that $\nu$ has a density, though Borell's original result is slightly more general. Note that the cases where $a$ or
$b$ belong to $\left\{ -\infty,0,+\infty\right\} $ are interpreted
in the limiting sense: We define $\left((1-\lambda)x^{a}+\lambda y^{a}\right)^{1/a}$
to be $\max(x,y)$ if $a=\infty$, $x^{1-\lambda}y^{\lambda}$ if
$a=0$, and $\min(x,y)$ if $a=-\infty$. 

Of special interest to us will be the case $a=b=0$, when Borell's
theorem essentially reduces to the Prékopa\textendash Leindler inequality (\cite{Prekopa71, Leindler72}):
A measure $\nu$ is \emph{log-concave}, i.e. it satisfies 
\begin{equation}
\nu\left((1-t)K_{0}+t K_{1}\right)\ge\nu(K_{0})^{1-t}\nu(K_{1})^{t},
\end{equation}
 if and only if its density $\phi$ is of the form $\phi=e^{-V}$
where $V:\RR^{n}\to[0,\infty)$ is a convex function.

Borell's theorem gives us the best possible Brunn\textendash Minkowski
inequality satisfied by a measure $\nu$ for \emph{all} Borel sets
$K_{0}$ and $K_{1}$. However, a better inequality may hold for specific
choices of $K_{0}$ and $K_{1}$. For example, let $\gamma$ denote the standard
Gaussian measure on $\RR^{n}$, i.e. the measure with density $\frac{\dd\gamma}{\dd x}=\frac{1}{\left(2\pi\right)^{n/2}}e^{-\left|x\right|^{2}/2}$.
This is a log-concave measure, but in general it doesn't satisfy the
Brunn\textendash Minkowski inequality with any exponent $a>0$. However,
Kolesnikov\textendash Livshyts (\cite{KolesnikovLivshyts21}) proved that
if $K_{0},K_{1}\subseteq\RR^{n}$ are convex sets such that $0\in K_{0}\cap K_{1}$,
then we have 
\begin{equation} 
\gamma\left((1-t)K_{0} + t K_{1}\right)^{\frac{1}{2n}}\ge(1-t)\gamma(K_{0})^{\frac{1}{2n}}+t \gamma(K_{1})^{\frac{1}{2n}}.\label{eq:gassian-BM-2n}
\end{equation}

Eskenazis and Moschidis (\cite{EskenazisMoschidis21}) further proved that if $K_{0},K_{1}\subseteq\RR^{n}$
are convex and \emph{origin-symmetric} (i.e. $K_{0}=-K_{0}$ and $K_{1}=-K_{1}$),
the exponent $\frac{1}{2n}$ improves to $\frac{1}{n}$, which is
the best possible. This settled a conjecture of Gardner and Zvavitch
(\cite{GardnerZvavitch10}). In fact, the original conjecture was that perhaps
the condition $0\in K_{0}\cap K_{1}$ is sufficient to have (\eqref{eq:gassian-BM-2n})
with exponent $\frac{1}{n}$, but Nayar and Tkocz (\cite{NayarTkocz13})
proved that this is not the case and suggested that central symmetry
may be the correct condition.

More generally, the following conjecture is widely believed:
\begin{con}
\label{conj:DBM}Let $\nu$ be an even log-concave measure on $\RR^{n}$.
Then for every $K_{0},K_{1} \subseteq\RR^{n}$ which are convex and origin-symmetric
we have 
\begin{equation}
\nu\left((1-t)K_{0}+t K_{1}\right)^{\frac{1}{n}}\ge(1-t)\nu(K_{0})^{\frac{1}{n}}+ t\nu(K_{1})^{\frac{1}{n}}.
\end{equation}
\end{con}

One reason to believe this conjecture is the paper \cite{LivshytsMarsigliettiNayarZvavitch17},
which shows that it would follow from another important conjecture,
the log-Brunn\textendash Minkowski conjecture. In particular, the
results of Böröczky\textendash Lutwak\textendash Yang\textendash Zhang
(\cite{BoroczkyLutwakYangZhang12}) show that Conjecture \ref{conj:DBM} holds
in dimension $n=2$. 

In dimensions $n>2$, not much is known. It is immediate from Borell's
theorem that uniform measures on convex bodies satisfy the Brunn\textendash Minkowski
inequality with exponent $\frac{1}{n}$ without any symmetry assumptions.
The results of \cite{Cordero-ErasquinRotem23} show that the conjecture
holds in particular for rotation-invariant log-concave measures (as
well as some rotation-invariant measures that are not log-concave). Livshyts
(\cite{Livshyts23}) proved that the conjecture holds for all even
log-concave measures, but with the optimal exponent $\frac{1}{n}$
replaced by $\frac{1}{n^{4+o(1)}}$. 

The proofs of all the results mentioned so far followed the same general
strategy. The idea was developed by Kolesnikov and Milman (\cite{KolesnikovMilman18},
\cite{KolesnikovMilman22}) to study various Brunn\textendash Minkowski
type inequalities, both on $\RR^{n}$ and in the setting of Riemannian
manifolds. The details will not be very important for us, but roughly
speaking one tries to prove that $\frac{\dd^{2}}{\dd t^{2}}\left(\nu\left((1-t)K_{0}+t K_{1}\right)^{a}\right)\le0$
by directly computing this second derivative. Without loss of generality
it is enough to check this at $t=0$, which gives a functional inequality
for functions $f:\partial K_{0} \to\RR$. Then the crucial step is to transform
this inequality to an inequality for functions $u:K_{0} \to\RR$ by taking
$u$ to be the solution of a certain elliptic PDE with $f$ as its
Neumann boundary condition. As far as we can tell this strategy strongly
depends on the convexity of $K_{0}$ and $K_{1}$, and without this assumption
no results are possible. 

\subsection{Weighted Brunn--Minkowski inequalities for star bodies}
Despite the difficulties mentioned above when dealing with non-convex sets, in this paper we prove the following theorem:
\begin{thm}
\label{thm:main-BM intro}Let $V:\RR^{n}\to[0,\infty)$ be a $p$-homogeneous
convex function for $1<p<\infty$, and let $\nu$ be a measure with
density $\frac{\dd\nu}{\dd x}$ proportional to $e^{-V}$. Let $K_{0},K_{1}\subseteq\RR^{n}$
be star bodies. Then for all $0\le t \le1$ we have
\begin{equation}
\text{\ensuremath{\nu\left((1-t)K_{0}+t K_{1}\right)^{\frac{p-1}{pn}}\ge(1-t)\nu(K_{0})^{\frac{p-1}{pn}}+ t \nu\left(K_{1}\right)^{\frac{p-1}{pn}}.}}
\end{equation} 
\end{thm}

To explain the terminology, we say that $K\subseteq\RR^{n}$ is a star
body if $K$ is a Borel set such that for all $x\in K$ and $0\le \lambda \le1$
we have $\lambda x\in K$ (i.e. the ``star'' of $K$ is always assumed
to be at the origin). Also recall that $V$ is $p$-homogeneous if
$V(\lambda x)=\lambda^{p}V(x)$ for all $x\in\RR^{n}$ and $\lambda>0$. 

Note that one specific choice, $V(x)=\frac{1}{2}\left|x\right|^{2}$,
gives the Kolesnikov\textendash Livshyts result (\eqref{eq:gassian-BM-2n}),
but this time without the convexity assumption on the sets. For any
fixed $p>1$ the exponent $\frac{p-1}{pn}$ is asymptotically (for
large $n$) better than the result of \cite{Livshyts23}, and again
we need no convexity or symmetry assumptions, but of course our result
is only for homogeneous potentials and not for all log-concave measures.
Note that as $p\to\infty$ we have $\frac{p-1}{pn}\to\frac{1}{n}$,
the optimal exponent. This makes sense since for large $p$ the measure
$\nu$ is almost uniform on some convex body, and such uniform measures
do satisfy the Brunn\textendash Minkowski inequality with exponent
$\frac{1}{n}$, applicable to all Borel sets. 

As mentioned above, as far as we know none of the previously used
techniques can prove a result like Theorem \ref{thm:main-BM intro}. Our
proof will therefore follow completely different lines, using tools
of optimal transport and information theory. We lift the problem at hand from the realm of geometry of sets to the realm of geometry of measures, where calculus is sometimes ``smoother'' than in classical differential geometry. Arguably, even for the Gaussian measure, our proof is simpler than the known arguments.

The theory of optimal transport gives us a natural way to interpolate between two probability measures $\mu_{0}$ and $\mu_{1}$, such that the interpolant $\mu_{t}$ is supported in the Minkowski average $(1-t) \cdot \supp(\mu_{0}) + t \cdot \supp(\mu_{1})$ of the two supports, thereby playing a role parallel to the Minkowski averaging of sets. Along these interpolations, the relative entropy $D(\cdot \Vert \nu)$ enjoys convexity properties of various degrees depending on $\nu$. The reader will find precise definitions in Section \ref{sec: overview}. However, the main point is that a variational formula for $\nu$ can be written in terms of the relative entropy $D(\cdot \Vert \nu)$,
\begin{equation} \label{eq: variationalmeasureIntro}
\sup_{\mu \in \mathcal{P}(K)} e^{- D(\mu \Vert \nu)} = \nu (K),
\end{equation}
where the supremum runs over all probability measures on $K$ and is attained at the normalized restriction $\nu_{K}$ of $\nu$ to $K$, allowing us to translate convexity properties of the relative entropy into concavity properties of the measure $\nu$.

Since the density $\frac{\d \nu_{K}}{ \d \nu}$ is radially decreasing when $K$ is a star body, we will obtain Theorem \ref{thm:main-BM intro} as a corollary to the following result.
\begin{thm}
\label{thm:main-entropy}Let $V:\RR^{n}\to[0,\infty)$ be a $p$-homogeneous
convex function for $1<p<\infty$, and let $\nu$ be the probability
measure with density proportional to $e^{-V}$. Let $\mu_{0}$, $\mu_{1}$
be Borel probability measures, absolutely continuous with respect
to $\nu$, such that $\frac{\dd\mu_{0}}{\dd\nu}$ and $\frac{\dd\mu_{1}}{\dd\nu}$
are radially decreasing. Finally, let $\left\{ \mu_{t}\right\} _{t\in[0,1]}$
be the displacement interpolation between $\mu_{0}$ and $\mu_{1}$.
Then 
\begin{equation}
t\mapsto e^{-\frac{p-1}{pn}\DD{\mu_{t}}{\nu}}
\end{equation}
 is concave. 
\end{thm}

It is well known that the behavior of information-theoretic quantities under optimal transport is intimately related to the geometry of the underlying space, which is useful in the context of Brunn\textendash Minkowski type inequalities (see  for example Remark \ref{rem: cdhistory}). There are known proofs of the classical Brunn\textendash Minkowski or Prékopa\textendash Leindler
inequalities along these lines (\cite[Chapter 6]{Villani03}). However, to the best of
our knowledge this is the first time such tools are used to prove
previously unknown Brunn\textendash Minkowski type inequalities for measures on Euclidean space. 

To prove Theorem \ref{thm:main-entropy} we split the relative entropy $D(\mu \Vert \nu)$ of a measure $\mu$ into its ``internal energy'' $D(\mu \Vert \vol_{n})$ and ``potential energy'' $\int V \d \mu$ components. We observe that, under the assumptions on $\mu_{0}$ and $\mu_{1}$ as in Theorem \ref{thm:main-entropy}, $e^{- \frac{p-1}{n} \int V \d \mu_{t}}$ is concave in $t$. Theorem \ref{thm:main-entropy} is then obtained by combining our observation with a result of Erbar, Kuwada, and Sturm \cite{ErbarKuwadaSturm15} that $e^{- \frac{1}{n} D(\cdot \Vert \vol_{n})}$ is concave along all displacement interpolations. The aforementioned convexity property of the potential energy is established locally, that is, by taking two derivatives. 

\subsection{Towards the best possible exponent in the entropy inequality}

As mentioned, Theorem \ref{thm:main-BM intro} implies the Kolesnikov-Livshyts inequality \eqref{eq:gassian-BM-2n}, but not the result of Eskenazis--Moschidis with the optimal exponent $\frac{1}{n}$. Since this result is 
known to be true, one can hope for a version of Theorem \ref{thm:main-entropy} that will imply it, that is, a version with the sharp coefficient $\frac{1}{n}$, under suitable assumptions on the potential $V$ and the measures $\mu_0$ and $\mu_1$. 

One natural way to approach this problem is to perform the local analysis directly on $D(\cdot \Vert \nu)$ instead of splitting it to two components. Indeed, the concavity of $e^{-a D(\mu_{t} \Vert \nu)}$ along a displacement interpolation $\{ \mu_{t} \}_{t \in [0,1]}$ amounts to the inequality
\begin{equation}
    \frac{\d^{2}}{\d t^{2}} D(\mu_{t} \Vert \nu) \geq a \left( \frac{\d}{\d t} D(\mu_{t} \Vert \nu)  \right)^{2}.
\end{equation}
As we will see later, both sides of the above inequality can be explicitly written in terms of certain differential operators applied to the velocity field driving the optimal transport. This  leaves us with an inequality of almost the same form as the sufficiency criteria of Kolesnikov--Livshyts discussed earlier. 
In the case of $\nu = \gamma$, an auxiliary construction from \cite{EskenazisMoschidis21} allows us to prove the following theorem:
\begin{thm} \label{thm: gaussianentintro}
    Suppose $\{ \mu_{t} \}_{t \in [0,1]}$ is a displacement interpolation such that each $\mu_{t}$ is even and satisfies the odd Poinc\'are inequality with constant $1$, that is $\int f^{2} \d \mu_{t} \leq \int \vert \grad f \vert^{2} \d \mu$, for every odd test function $f$. Then,
    \begin{equation}
    \frac{\d^{2}}{\d t^{2}} D(\mu_{t} \Vert \gamma) \geq 2 \wasser^{2}(\mu_{0},\mu_{1}) + \frac{1}{n} \left( \frac{\d}{\d t} D(\mu_{t} \Vert \gamma)  \right)^{2},
    \end{equation}
    where $\wasser(\mu_{0},\mu_{1})$ denotes the $2$-Wasserstein metric (induced by the optimal transport problem).
\end{thm}
Thus, as readers familiar with the curvature-dimension framework may notice, under these assumptions on the $\mu_{t}$, not only can the ``dimensional'' improvement to $\frac{1}{n}$ be obtained, but the ``curvature'' also improves to $2$. Note that, without any restrictions on the $\mu_{t}$, $\frac{\d^{2}}{\d t^{2}} D(\mu_{t} \Vert \gamma) \geq  \wasser^{2}(\mu_{0},\mu_{1})$ is always true, which is reflective of the fact that Gauss space satisfies properties akin to a complete Riemannian manifold with Ricci curvature lower bounded by the metric tensor.

When $K$ is a convex body, then the density of $\nu_{K}$ with respect to $\nu$ is a log-concave function. This is the property of $\nu_{K}$ that typically expresses itself in extensions of various phenomena from the geometry of sets to the geometry of measures. When $\nu$ is the Gaussian measure $\gamma$, measures having a log-concave density with respect to $\gamma$ are called strongly log-concave measures. Given that the Gaussian measure satisfies a Brunn--Minkowski inequality with exponent $\frac{1}{n}$ over origin-symmetric convex bodies, it is a natural question whether displacement interpolation between two even strongly log-concave probability measures satisfies the assumptions of Theorem \ref{thm: gaussianentintro}.
\begin{qstn} \label{qstn: poincare}
    Suppose $\mu_{0}, \mu_{1}$ are even strongly log-concave probability measures on $\R^n$. Let $\{ \mu_{t} \}_{t \in [0,1]}$ denote the displacement interpolation between $\mu_{0}$ and $\mu_{1}$. Is it true that each $\mu_{t}$, $t\in (0,1)$, satisfies the Poinc\'are inequality for odd functions with constant $1$?
\end{qstn}

Note that the Brascamp--Lieb inequality (\cite[Theorem 4.1]{BrascampLieb76}) guarantees that $\mu_{0}$ and $\mu_{1}$ in the above setup satisfy a Poinc\'are inequality with constant $1$ for all functions. Also note, that a positive answer to Question \ref{qstn: poincare}, combined with our Theorem \ref{thm: gaussianentintro}, would yield an entropic proof of the Eskenazis--Moschidis result. 
We show that Question \ref{qstn: poincare} has an affirmative answer in several cases, namely when the dimension of the ambient space is one, or when both $\mu_{0}$ and $\mu_{1}$ are themselves Gaussian measures, or when one of $\{ \mu_{0}, \mu_{1} \}$ is the standard Gaussian $\gamma$. 


\subsection{Functional Inequalities}
While we cannot answer Question \ref{qstn: poincare} in full generality, the case where $\mu_0$ or $\mu_1$ is $\gamma$ itself does suffice to prove several fundamental functional inequalities.

The HWI inequality, originally introduced by Otto and Villani in \cite{OttoVillani00}, asserts a relationship between the three elemental quantities--- $D(\cdot \Vert \gamma), \wasser(\cdot, \gamma)$, and $I(\cdot \Vert \gamma)$. The ``H'' in HWI stands for relative entropy (as it is sometimes denoted), ``W'' is for the Wasserstein distance, and ``I'' for Fisher information. Recall that, if $P^{\ast}_{s} \mu$ denotes the evolution of the measure $\mu$ under the Ornstein--Uhlenbeck process at time $s$, then the Fisher information is $I(\mu \Vert \gamma) = - \left. \frac{\d}{\d s} D(P^{\ast}_{s} \mu \Vert \gamma) \right|_{s=0}$. Using the methods from \cite{ErbarKuwadaSturm15}, we obtain the following strengthening of the Gaussian HWI inequality for even strongly log-concave measures as a consequence of our results. 
\begin{thm}[An HWI inequality] \label{thm: HWIintro}
    Let $\mu$ be an even strongly log-concave probability measure on $\R^n$. Suppose $\{ \mu_{0}, \mu_{1} \} = \{ \mu , \gamma \}$. Then,
        \begin{equation} \label{eq: HWIintro}
        e^{\frac{D(\mu_{0} \Vert \gamma)}{n} - \frac{D(\mu_{1} \Vert \gamma)}{n}} \leq \cos \left( \sqrt{\frac{2}{n}}\wasser(\mu_{0}, \mu_{1}) \right) + \frac{1}{\sqrt{2n}} \sin \left(  \sqrt{\frac{2}{n}} \wasser(\mu_{0}, \mu_{1}) \right) \sqrt{I(\mu_{0} \Vert \gamma)}.
        \end{equation}
\end{thm}

The original HWI inequality of Otto--Villani (\cite[Theorem 3]{OttoVillani00}) for the Gaussian measure, which is valid without symmetry or log-concavity assumptions on $\mu$, states that
\begin{equation} \label{eq: originalHWI}
D(\mu_{0} \Vert \gamma) - D(\mu_{1} \Vert \gamma) \leq \wasser (\mu_{0}, \mu_{1}) \sqrt{I(\mu_{0} \Vert \gamma)} - \frac{1}{2} \wasser^{2} (\mu_{0}, \mu_{1}).
\end{equation}
To compare our inequality \eqref{eq: HWIintro} with the original HWI inequality \eqref{eq: originalHWI}, the reader may raise both sides of \eqref{eq: HWIintro} to the power $n$ and send $n \to \infty$. A key feature of HWI type inequalities, perhaps even the motivation behind it in \cite{OttoVillani00}, is that they unify two other important families of inequalities. At one end, by setting $\mu_{0} = \mu$ and $\mu_{1} = \gamma$ in Theorem \ref{thm: HWIintro}, after some elementary manipulations, we get a refinement of the Gaussian Logarithmic Sobolev inequality. Although first popularized in the mathematical community through a contribution of Gross (\cite{Gross75}), the Gaussian Logarithmic Sobolev inequality had previously appeared in the work of Federbush (\cite{Federbush69}) and even earlier as an information-theoretic uncertainty principle by Stam (\cite{Stam59}). 
 \begin{cor}[A Logarithmic Sobolev inequality with an application] 
     Let $\mu$ be an even strongly log-concave probability measure on $\R^n$. Then, 
     \begin{equation}
4 D(\mu \Vert \gamma) \leq 2n \left( e^{\frac{2}{n}D(\mu \Vert \gamma)} - 1 \right) \leq I( \mu \Vert \gamma).
       \end{equation}
       Therefore, 
       \begin{equation}
 D(P^{\ast}_{t} \mu \Vert \gamma ) \leq e^{-4t} D(\mu_{0} \Vert \gamma).
       \end{equation}
 \end{cor}
On the other hand, with $\mu_{0} = \gamma$ and $\mu_{1} = \mu$, the transportation cost-information inequality of Talagrand (\cite{Talagrand96}) is improved for even strongly log-concave probability measures.
\begin{cor}[A Talagrand inequality] 
    Let $\mu$ be an even strongly log-concave probability measure on $\R^n$. Then,
        \begin{equation}
\wasser^{2} (\mu , \gamma) \leq - n \log \cos \left( \sqrt{\frac{2}{n}} \wasser(\mu, \gamma) \right) \leq D(\mu \Vert \gamma).
        \end{equation}
\end{cor}
The original inequalities assert $2D (\mu \Vert \gamma) \leq I(\mu \Vert \gamma)$ and $\frac{1}{2} \wasser^{2}(\mu, \gamma) \leq D(\mu \Vert \gamma)$, respectively, for all probability measures $\mu$. Thus, even if we completely ignore the dimensional improvements that we have obtained, both the Gaussian Logarithmic Sobolev and the Talagrand inequalities show an improvement by a multiplicative factor of $2$ when restricted to even strongly log-concave measures. From examples presented in the main body of the text (see Examples \ref{exmpl: betterlsineedsslc} and \ref{exmpl: bettertalagrandneedsslc}), the reader will see that both assumptions--- central symmetry and strong log-concavity of the measure, are essential for this improvement. Dimensional improvements in Gaussian functional inequalities have been pursued before, for example, in the works of Bakry and Ledoux (\cite{BakryLedoux06}), Bobkov, Gozlan, Roberto, and Samson (\cite{BobkovGozlanRobertoSamson14}), and more recently, Bolley, Gentil, and Guillin (\cite{BolleyGentilGuillin18}). However, none of these results gives the multiplicative improvement we get by restricting to even strongly log-concave measures. Readers interested in dimension-dependent functional inequalities may also consult \cite{BakryBolleyGentil12, BakryBolleyGentil17, EskenazisShenfeld24, Shenfeld24}. 

The rest of this paper is organized as follows. In Section \ref{sec: overview} we
give the necessary background on optimal transport and describe
the previous works on displacement concavity of entropy. In Section
\ref{sec: dimconhomo} we give the relatively short proofs of Theorems \ref{thm:main-BM intro} and \ref{thm:main-entropy}.

Then, in Section \ref{sec: gaussian} we restrict our attention to the case $\nu=\gamma$, to discuss potential improvements to convexity properties of $D(\cdot \Vert \gamma)$ beyond those established in Section \ref{sec: dimconhomo}. In particular we prove Theorem \ref{thm: gaussianentintro}. Finally, in Section \ref{sec: gaussianapp} we apply our findings from the previous section to prove functional inequalities such as Theorem \ref{thm: HWIintro} and its corollaries. 

\subsection*{Some notation and conventions:}
\begin{itemize}
\item The standard Euclidean norm on $\R^n$ is denoted by $\vert \cdot \vert$ .
    \item Let $K \subseteq \R^n$ be a measurable subset. We denote by $\mathcal{P}(K)$ the space of all Borel probability measures on $K$, $\mathcal{P}_{2}(K)$ all Borel probability measures on $K$ with finite second moments, and $\mathcal{P}_{ac}(K)$ the space of Borel probability measures on $K$ absolutely continuous with respect to the Lebesgue measure $\vol_{n}$ on $\R^n$. Further, we set $\mathcal{P}_{2,ac}(K) = \mathcal{P}_{2}(K) \cap \mathcal{P}_{ac}(K)$.
    \item We will sometimes write $\int f$ for $\int f(x) \d x$. Lebesgue measure is the only case when the measure may not be specified in the integral.
    \item Let $\nu$ be a Borel measure on $\R^n$. For every Borel set $K$ such that $\nu(K) >0$, we define the probability measure $\nu_{K}$ by setting $\nu_{K}(E) = \frac{\nu(E \cap K)}{\nu(K)}$, for every $E$. 
    \item An \emph{even function} is a function $f$ satisfying $f(x) = f(-x)$ for all $x$. An \emph{odd function} is a function satisfying $f(-x) = - f(x)$ for all $x$. An \emph{even measure} is a measure $\mu$ on $\R^n$ satisfying $\mu(E) = \mu (-E)$ for all measurable sets $E$. 
\end{itemize}

\paragraph{\textbf{Acknowledgments:}} We would like to thank Emanuel Milman for providing some useful references regarding the Logarithmic Sobolev inequality. We would also like to thank the anonymous referees for their very useful suggestions and corrections. The authors were supported by ISF grant 1468/19 and NSF-BSF grant 2022707.
\paragraph{\textbf{Declarations of interest:}} none.

\section{Optimal transport, entropy, and an overview of our method} \label{sec: overview}
While the general problem of optimal transport can be posed in a much more abstract setting, we will focus on the setup relevant for our purpose, namely the optimal transport problem in $\R^n$ with quadratic cost. For details beyond the short treatment we present below, the reader is referred to the excellent textbook \cite{Villani03}, especially Chapters 2 and 5. 

For $\mu_{0}, \mu_{1} \in \mathcal{P}_{2}(\R^n)$, denote by $\Pi (\mu_{0}, \mu_{1})$ the collection of all couplings between $\mu_{0}$ and $\mu_{1}$. This means, $\Pi (\mu_{0}, \mu_{1})$ consists of probability measures $\pi \in \mathcal{P}(\R^n \times \R^n)$ such that,
\begin{equation}
\pi (A \times \R^n) = \mu_{0}(A), \, \pi (\R^n \times B) = \mu_{1}(B),
\end{equation}
for all Borel sets $A,B \subseteq \R^n$. If $X_{0},X_{1}$ are $\R^n$-valued random vectors with distributions $\mu_{0}, \mu_{1}$, respectively, then $\Pi (\mu_{0}, \mu_{1})$ is the collection of all possible joint distributions for $(X_{0}, X_{1})$. 

The optimal transport problem, with quadratic cost, asks to minimize
\begin{equation}
I[\pi] = \int  \vert x - y \vert^{2} \d \pi (x,y),
\end{equation}
over all $\pi \in \Pi(\mu_{0}, \mu_{1})$. This models the situation where mass distributed according to $\mu_{0}$ in space is to be displaced to mass distributed according to $\mu_{1}$, where $\vert x-y \vert^2$ is the cost of moving a unit mass from the point $x$ to the point $y$. A coupling $\pi$ represents a plan, where the infinitesimal $\d \pi (x,y)$ is the amount of mass to be moved from $x$ to $y$. The problem then asks to compute the cheapest way to displace $\mu_{0}$ to $\mu_{1}$.

The optimal cost, 
\begin{equation}
\wasser(\mu_{0}, \mu_{1}) = \left( \inf_{\pi \in \Pi (\mu_{0}, \mu_{1})} \int  \vert x - y \vert^{2} \d \pi (x,y) \right)^{1/2},
\end{equation}
called the \emph{Wasserstein distance of order $2$} between $\mu_{0}$ and $\mu_{1}$, is in fact a metric on $\mathcal{P}_{2} (\R^n)$. The resulting metric space is called the \emph{$2$-Wasserstein space over $\R^n$}, denoted $\wasser (\R^n)$. The Wasserstein space over $\R^n$ is a geodesic space, that is, every two points can be connected by a continuous path $c:[0,1] \to \wasser (\R^n)$ such that, $\wasser(c(s),c(t)) = \vert t-s \vert \wasser (c(0),c(1))$ for any intermediate times $s,t \in [0,1]$. Every optimal coupling $\pi \in \Pi (\mu_{0}, \mu_{1})$ gives rise to a geodesic via $c_{\pi} (t) = (\sigma_{t})_{\#} \pi$, where $\sigma_{t}: (x,y) \mapsto (1-t)x + t y$. Notions related to geodesic-convexity in $\wasser (\R^n)$ are given special names, as in the definition below.

\begin{definition}
   \begin{enumerate}
   \item[]
       \item A subset $S \subseteq \wasser (\R^n)$ is called a \emph{displacement convex set}, if given any two measures $\mu_{0}, \mu_{1} \in S$, 
       \begin{equation}
\Big\{ (\sigma_{t})_{\#}\pi \, : \, \wasser^{2} (\mu_{0}, \mu_{1}) = \int \vert x-y \vert^{2} \d \pi (x,y), t \in [0,1] \Big\} \subseteq S.
       \end{equation}
       \item A map $F: S \to \R \cup \{ + \infty \}$ is called a \emph{displacement convex functional} if the map $t \mapsto F ((\sigma_{t})_{\#}\pi )$ is convex on $[0,1]$, for every $\pi$ such that $ \wasser^2 (\mu_{0}, \mu_{1}) = \int \vert x-y \vert^{2} \d \pi (x,y)$. Similarly, $F$ is called \emph{displacement concave} if  $-F$ is a displacement convex functional.
       \item If there is a unique constant-speed geodesic $\{ \mu_{t} \}_{t \in [0,1]}$ from $\mu_{0}$ to $\mu_{1}$, then it is called the \emph{displacement interpolation} between $\mu_{0}$ and $\mu_{1}$, and each $\mu_{t}$ is called a \emph{displacement interpolant}.
   \end{enumerate}
\end{definition}

The geodesically convex subset $\mathcal{P}_{2,ac}(\R^n) \subseteq \wasser(\R^n)$ is uniquely geodesic. This is facilitated by a theorem of Brenier (\cite{Brenier91}), which ensures a unique solution to optimal transport problem if $\mu_{0}$ has density, and describes the structure of the optimal coupling. 
\begin{thm}\cite[Theorem 2.12 (ii)]{Villani03} \label{thm: Brenier}
    Let $\mu_{0}, \mu_{1} \in \mathcal{P}_{2}(\R^n)$. If $\mu_{0} \in \mathcal{P}_{2,ac} (\R^n)$ in addition, then there exists a unique convex function $\phi$ such that, $T = \grad \phi$ pushes forward $\mu_{0}$ to $\mu_{1}$. Moreover, in this case the optimal coupling in the optimal transport problem is unique, and is given by $\pi = \left[ x \mapsto (x, T(x)) \right]_{\#} \mu_{0}$. 
\end{thm}
The map $T$ in the above theorem is called the \emph{Brenier map from $\mu_{0}$ to $\mu_{1}$}. If the conditions of the above theorem are met, then we can explicitly write the displacement interpolant between $\mu_{0}$ and $\mu_{1}$ as $\mu_{t} = (T_{t})_{\#} \mu_{0}$, where $T_{t} = (1-t) I + t T$ linearly interpolates between the identity map and the Brenier map from $\mu_{0}$ to $\mu_{1}$. This interpolation was introduced in the pioneering work of McCann (\cite{McCann97}), where the fact that the Jacobian matrix of the \emph{transport maps} $T_{t}$ takes values in the space of positive semidefinite matrices proved to be useful in establishing displacement convexity properties of some important functionals. For example, \cite[Theorem 2.2]{McCann97} contains the following functionals as special cases. 
\begin{exmpl} \label{exmpl: sbentropy}
The \emph{Shannon--Boltzmann entropy} $h(\mu) = - \int \left( \frac{\d \mu}{\d x} \right) \log \left( \frac{\d \mu}{\d x} \right) \d x$ on $\mathcal{P}_{2,ac}(\R^n)$ is displacement concave. In fact, this functional can be meaningfully extended to all of $\wasser (\R^n)$ by setting $h(\mu) = - \infty$, when $\mu$ does not have density with respect to the Lebesgue measure.
\end{exmpl}
\begin{exmpl} \label{exmpl: renyientropy}
    The \emph{exponentiated R\'enyi entropy} $A(\mu) = e^{\frac{1}{n} h_{(n-1)/n}(\mu)} := \int \left( \frac{\d \mu}{\d x} \right)^{1 - \frac{1}{n}} \d x$ is displacement concave.
\end{exmpl}
Much later, in a considerably broader framework than we need here, Example \ref{exmpl: sbentropy} was strengthened by Erbar, Kuwada, and Sturm (\cite{ErbarKuwadaSturm15}) in the form below. 
\begin{exmpl} \label{exmpl: expent}
    The exponentiated entropy $e^{h(\mu)/n}$ is displacement concave on $\wasser(\R^n)$. 
\end{exmpl}

Since $\supp (\mu_{t}) \subseteq (1-t) \supp (\mu_{0}) + t \supp(\mu_{1})$, displacement concavity can be used to obtain sumset inequalities for geometric quantities that admit a variational interpretation in terms of probability measures. For example, if $K \subseteq \R^n$ is a compact set, then 
\begin{equation}
\sup_{\mu \in \mathcal{P}_{ac}(K)} h(\mu) = \log \vol_{n} (K), \, \sup_{\mu \in \mathcal{P}_{ac}(K)} A(\mu) = \vol_{n}(K)^{1/n}.
\end{equation}
The upper-bounds implicit in both these formulas can be proved by a straightforward application of Jensen's inequality, then one can check that these upper-bounds are attained by the uniform distribution on $K$, leading to the variational expressions above. Suppose $K_{0}$ and $K_{1}$ are compact subsets of $\R^n$. Then, Example \ref{exmpl: sbentropy} implies the $0$-concavity of the Lebesgue measure, while Example \ref{exmpl: renyientropy} and Example \ref{exmpl: expent} imply the (a priori superior) $(1/n)$-concavity of the Lebesgue measure. 

In this paper, we appropriately generalize Example \ref{exmpl: expent}. To lift the framework above to deal with an arbitrary reference measure instead of the Lebesgue measure, we now introduce a way to measure entropy with respect to a reference measure. 
\begin{definition}[Relative entropy]
    Let $\nu$ be a $\sigma$-additive Borel measure on $\R^n$. We define the following functional on $\wasser(\R^n)$, 
    \begin{equation}
D(\mu \Vert \nu) =
\begin{cases}
\int \left( \frac{\d \mu}{\d \nu} \right) \log \left( \frac{\d \mu}{\d \nu} \right) \d \nu, & \textnormal{ if } \mu \textnormal{ has density w.r.t. } \nu , \\
+ \infty, & \textnormal{ otherwise, }  \\
\end{cases}
\end{equation}
called the \emph{relative entropy of $\mu$ with respect to $\nu$}.
\end{definition}
\begin{rem}
 The relative entropy $D(\mu \Vert \nu)$ is a quantification of how much $\mu$ is ``spread out'' from the viewpoint of $\nu$. In $\R^n$, we get an absolute measure of spread by looking at negative the amount $\mu$ is spread out from the most spread out measure $\vol_{n}$, $h( \mu) = - D (\mu \Vert \vol_{n})$. 
\end{rem}
Using $D(\cdot \Vert \nu)$, one can obtain a variational formulation of the $\nu$-measure, exactly as in the case of the Shannon--Boltzmann entropy and volume. 
\begin{lem} \label{lem: numeasureinent}
    Let $K \subseteq \R^n$ be a compact set, then 
    \begin{equation}
        \inf_{\mu \in \mathcal{P}(K)} D (\mu \Vert \nu) = - \log \nu (K),
    \end{equation}
    where the supremum is attained by $\nu_{K}$. 
\end{lem}
\begin{proof}
    Let $X$ be an $\R^n$-valued random vector with distribution $\mu \in \mathcal{P}(K)$. Because of the way relative entropy is defined, we only need to consider measures $\mu$ having density, say $f$, with respect to $\nu$. Using Jensen's inequality we obtain,
    \begin{equation}
    - D(\mu \Vert \nu) = - \int f \log f \d \nu = \mathbb{E} \log \frac{1}{f(X)} \leq \log \mathbb{E} \frac{1}{f(X)} \leq \log \nu (K). 
    \end{equation}
    The bound is attained when $f$ is constant, that is, when $\mu = \nu_{K}$. This can be verified directly.  
\end{proof}

Therefore, we have the following general principle formalizing our method. 
\begin{prop} \label{pro: wand}
    Let $\nu$ be a $\sigma$-finite Borel measure on $\R^n$. Suppose $S \subseteq \wasser(\R^n)$ is a displacement convex set containing $\nu_{K}$, for all compact sets in a class $\mathcal{K}$ of compact subsets of $\R^n$. If $e^{- a D(\cdot \Vert \nu)}$ is displacement concave on $S$, $a>0$, then 
    \begin{equation} \label{eq: wand}
    \nu ((1-t) K_{0} + t K_{1})^{a} \geq (1-t) \nu (K_{0})^{a} + t \nu (K_{1})^{a},
    \end{equation}
    for all $K_{0}, K_{1} \in \mathcal{K}$.
\end{prop}
\begin{proof}
    In the assumed displacement concavity inequality
    \begin{equation}
e^{-a D(\mu_{t} \Vert \nu )} \geq (1-t) e^{-a D(\mu_{0} \Vert \nu )} + e^{-a D(\mu_{1} \Vert \nu )},
    \end{equation}
    we plug in $\mu_{0} = \nu_{K_{0}}, \mu_{1} = \nu_{K_{1}}$. Since $\supp(\mu_{t}) \subseteq (1-t)K_{0} + tK_{1}$, we have $e^{-a D(\mu_{t} \Vert \nu )} \leq \nu((1-t)K_{0} + tK_{1})^{a}$ which concludes the proof. 
\end{proof}
\begin{rem}
    When $e^{-a D(\cdot \Vert \nu)}$, $a>0$, is displacement concave on all of $\wasser (\R^n)$, we get the $a$-concavity of the measure $\nu$. In this case, the converse is also true thereby extending Borell's characterization (Theorem \ref{thm: Borell}). This is contained, rather implicitly, in the work of Erbar, Kuwada, and Sturm (\cite{ErbarKuwadaSturm15}). The corresponding statement for $a=0$, namely the equivalence between displacement convexity of $D(\cdot \Vert \nu)$ and log-concavity of $\nu$, had been known earlier (see, for example \cite[Theorem 9.4.11]{AmbrosioGigliSavare08}).
\end{rem}
\begin{rem} \label{rem: cdhistory}
The displacement concavity property of $e^{-a D(\cdot \Vert \nu)}$ on $\wasser (\R^n)$, in the language of \cite{ErbarKuwadaSturm15}, is called the \emph{entropic $CD(0,1/a)$ condition} for the metric measure space $(\R^n, \Vert \cdot \Vert_{2}, \nu)$. In this expression, ``CD'' stands for \emph{curvature-dimension}. Given that a Riemannian manifold (equipped with its volume measure) has Ricci curvature bounded below by some $\kappa \in \R$ and its dimension is at most $n$, several geometric and functional inequalities (in terms of $\kappa$ and $n$) follow. The curvature-dimension condition for Markov diffusions, introduced in the work of Bakry--\'Emery \cite{BakryEmery85}, allows one to define the notion of ``Ricci curvature bounded below by $\kappa$, dimension bounded above by $n$'' for weighted Riemannian manifolds (that is, Riemannian manifolds equipped with a measure that is not necessarily the volume measure) admitting consequences akin to the non-weighted case. On the other hand, starting from the work of McCann \cite{McCann97}, through the papers of Otto--Villani \cite{OttoVillani00}, Cordero-Erausquin--McCann--Schmuckenschl\"ager \cite{Cordero-ErasquinMcCannSchmuckenschlager01}, it was eventually established in the work of von Renesse--Sturm \cite{vonRenesseSturm05} that the ``Ricci curvature bounded below by $\kappa$'' condition on weighted Riemannian manifolds could equivalently be described using the displacement $\kappa$-convexity of relative entropy on the $2$-Wasserstein space over the Riemannian manifold. In the landmark papers of Sturm \cite{Sturm06(i), Sturm06(ii)} and Lott--Villani \cite{LottVillani09}, not only were dimensional considerations added (using R\'enyi entropy) to complete the picture on the displacement convexity-based alternative to the Bakry--\'Emery approach to curvature-dimension in the weighted Riemannian setting, but the natural extension of this theory to a wide class of metric measure spaces was established and explored. In the setting of (\emph{infinitesimally Hilbertian}) metric measure spaces, the equivalence between the two approaches was established by Erbar--Kuwada--Sturm \cite{ErbarKuwadaSturm15}, in which the authors introduce and study the \emph{$(\kappa , n)$-convexity} (see Definition \ref{def: knconvexity}) of relative entropy on Wasserstein space. Particularly, in the setting of the Euclidean space equipped with a measure $\d \nu = e^{-V} \d x$, the $(\kappa, n)$-convexity of $D(\cdot \Vert \nu)$ on $\wasser(\R^n)$ is equivalent to the $CD(\kappa,n)$ condition defined in terms of the displacement convexity of (relative) R\'enyi entropy (perhaps better known as R\'enyi divergence among information theorists), which is further equivalent to the Bakry--\'Emery definition of $CD(\kappa, n)$ condition for the Markov semigroup generated by $Lf = \Delta f - \< \grad V , f \>$. However, such equivalences need not hold over specific choices of Wasserstein geodesics. In the present work, our choice of using exponentiated relative entropy instead of the exponentiated R\'enyi divergence plays a key role. 
\end{rem}

Sometimes, to prove the displacement convexity properties of certain functionals it is useful to take an ``Eulerian viewpoint''. For $\mu_{0}, \mu_{1} \in \mathcal{P}_{2,ac} (\R^n)$, consider the transport maps $T_{t}$ 
as in the discussion preceding Example \ref{exmpl: sbentropy}. The previous description of the maps $T_{t}$ expresses the displacement in terms of the particle trajectories $T_{t}(x)$, of each particle $x$; this is the ``Lagrangian viewpoint''. One can equivalently describe the displacement in terms of a time-dependent vector field $v_{t}$ such that,
\begin{equation}
v_{t} ( T_{t}(x)) = \frac{\d }{\d t} T_{t} (x).
\end{equation}
The reason why $v_{t}$ can be defined in the above manner is that trajectories in optimal transport do not cross (in an almost-everywhere sense). Further, an argument based on the dominated convergence theorem leads to the \emph{continuity equation}, 
\begin{equation}\label{eq: continuityeqn}
\frac{\d }{\d t} \int \phi \d \mu_{t} = \int \langle \grad \phi , v_{t} \rangle \d \mu_{t},
\end{equation}
for all test functions $\phi$. It can be shown that $v_{t}$ is of the form $\grad \theta_{t}$, and $\theta_{t}$ satisfies a Hamilton--Jacobi equation in the weak sense, 
\begin{equation}\label{eq: HamiltonJacobi}
\frac{\partial \theta_{t} }{\partial t} + \frac{\vert \grad \theta_{t} \vert^{2}}{2} = 0. 
\end{equation}
Roughly speaking, the fact that $v_{t}$ is a gradient field follows from Brenier's theorem and semigroup properties of optimal transport, while the Hamilton--Jacobi equation is a consequence of $T_{t}$ carrying each particle along a straight line. For more details, the reader is referred to \cite[Section 5.4]{Villani03}.

\section{Dimensional convexity properties of relative entropy for log-concave measures with homogeneous potential} \label{sec: dimconhomo}
In this section, we fix a measure $\nu$ of the form $\d \nu = e^{-V} \d x$, where $V: \R^n \to \R \cup \{ + \infty \}$ is a convex function. We will break down the proof of the main result of this section into several lemmas, for the convenience of the reader who may wish to generalize our results.

The relative entropy with respect to $\nu$ can be decomposed into two parts. Suppose $\mu$ is a probability measure that has density with respect to $\nu$. Write $\frac{\d \mu}{\d x} = f$. Then,
\begin{equation}
D(\mu \Vert \nu) = \int f \log \frac{f}{e^{-V}} \d x = \int f \log f \d x + \int Vf \d x = - h(\mu) + \mathcal{V}(\mu) ,
\end{equation}
where $\mathcal{V}(\mu) = \int V \d \mu$. Recall that we are interested in the concavity of 
\begin{equation}
e^{- a D(\cdot \Vert \nu)} = e^{a h(\cdot)} e^{- a \mathcal{V}(\cdot)}.
\end{equation}
We will deal with the two factors separately. 

Firstly, it is known that $e^{\frac{h(\cdot)}{n}}$ is displacement concave on $\wasser (\R^n)$. As mentioned earlier, this was obtained by Erbar, Kuwada, and Sturm \cite[Theorem 3.12]{ErbarKuwadaSturm15} in a much more general setup. For completeness we specialize their argument to outline a proof. 
\begin{thm}[\cite{ErbarKuwadaSturm15}] \label{thm: EKS15}
    Lebesgue measure on $\R^n$ satisfies the ``entropic $CD(0,n)$ condition'', that is, $e^{\frac{h(\cdot)}{n}}$ is displacement concave on $\wasser (\R^n)$.
\end{thm}
\begin{proof}[Proof outline]
    We want to show that the non-negative quantity $e^{\frac{h(\mu_{t})}{n}}$ is concave in $t$. If neither $\mu_{0}$ nor $\mu_{1}$ have density, there is nothing to prove. If one of them has density then the interpolants have density (\cite[Proposition 5.9]{Villani03}), and the result would follow from the corresponding result assuming both endpoints have density, using the upper-semicontinuity of the Shannon--Boltzmann entropy. Therefore we assume that $\mu_{0}, \mu_{1}$ have density with respect to the Lebesgue measure. Suppose $T_{t}= \grad \phi_{t}$ is the Brenier map from $\mu_{0}$ to $\mu_{t}$. Using the change of variables $f_{t}(\grad \phi_{t} (x))\det \grad^{2} \phi_{t} (x) = f_{0} (x)$ (see \cite[Theorem 4.8]{Villani03}) relating the densities $f_{t}$, and the $(1/n)$-concavity of the determinant on the space of positive-definite matrices, one obtains
    \begin{equation} \label{eq: dimensionaldensityrelation}
f_{t}((1-t)x + t y)^{-1/n} \geq (1-t)f_{0}(x)^{-1/n} + t f_{1}(x)^{-1/n},
    \end{equation}
    $\pi(x,y)-a.e.$, where $\pi$ is the optimal coupling between $\mu_{0}$ and $\mu_{1}$. We define a convex function,
    \begin{equation}
    G_{t}(a,b) = \log \left( (1-t) e^{a} + t e^{b}   \right).
    \end{equation}
    Then, by taking logarithms of both sides of the inequality \eqref{eq: dimensionaldensityrelation}, and integrating, 
    \begin{equation}
\begin{split}
    \frac{1}{n} h(\mu_{t}) &= \int - \frac{1}{n} \log f_{t} \d \mu_{t} \\
    &= \int -  \frac{1}{n} \log f_{t}((1-t)x + t y ) \d \pi (x,y)   \\
    & \geq \int \log \left(  (1-t)f_{0}(x)^{-1/n} + t f_{1}(y)^{-1/n}  \right) \d \pi (x,y) \\
    &= \int G_{t} \left( - \frac{1}{n} \log f_{0}(x), - \frac{1}{n} \log f_{1}(y) \right) \d \pi (x,y)\\
    & \geq G_{t} \left(  \int - \frac{1}{n} \log f_{0}(x) \d \pi (x,y) , \int - \frac{1}{n} \log f_{1}(y) \d \pi (x,y)    \right) \\
    &= G_{t} \left( \frac{1}{n}h(\mu_{0}) , \frac{1}{n}h(\mu_{1}) \right) = \log \left( (1-t)e^{h(\mu_{0})/n} + t e^{h(\mu_{1})/n}   \right). \qedhere
    \end{split}
    \end{equation}
\end{proof}
\begin{rem}
Let $X,Y$ be $\R^n$-valued random vectors. Then the above theorem says that
\begin{equation} \label{eq: EKS without scales}
e^{h(X+Y)/n} \geq e^{h(X)/n} + e^{h(Y)/n},
\end{equation}
where $(X,Y)$ are optimally coupled, that is, the distribution of $(X,Y)$ is the optimal coupling between the distributions of $X$ and $Y$. This can be seen by using the scaling property of entropy, $h(\lambda X) = h(X) + n \log \lambda$ for $\lambda \in (0, \infty)$. For a similar result, see the work of Aras and Courtade (\cite[Corollary 5]{ArasCourtade21}), where they give bounds for $h(X+Y)$ based on the amount of dependence allowed between $X$ and $Y$. At one end, when there is no restriction on the dependence, they obtain an inequality like the above. At the other end, when $X$ and $Y$ are assumed to be independent, they obtain the \emph{Entropy Power inequality}: $e^{2h(X+Y)/n} \geq e^{2h(X)/n} + e^{2h(Y)/n}$.
\end{rem}

We now turn to the other factor, $e^{-a \mathcal{V}(\cdot)}$. For this purpose, we would like a differential criterion for a function $e^{-a \theta (\cdot)}: [0,1] \to \R$ to be concave. 
\begin{lem} \label{lem: whatitmeansinderivative}
    Let $\theta : [0,1] \to \R$ be a twice continuously differentiable function, $a>0$. If we have 
    \begin{equation}
\theta ''(t) \geq a \left( \theta'(t) \right)^{2},
    \end{equation}
    then $e^{- a \theta(t)}$ is concave. 
\end{lem}
\begin{proof}
We simply use the second-derivative test for concavity, and the following calculation.
    \begin{equation}
    \left( e^{-a \theta } \right)'' = \left( - e^{- a \theta } a \theta ' \right)' = \left( - e^{a \theta} a \theta'' +  e^{-a \theta}a^{2} \left( \theta' \right)^{2}   \right) = a e^{-a \theta} \left(  a \left( \theta' \right)^{2} - \theta''  \right).
    \end{equation}
\end{proof}
Therefore, the quantities of interest are $\frac{\d}{\d t} \mathcal{V}(\mu_{t})$ and $\frac{\d^{2}}{\d t^{2}} \mathcal{V}(\mu_{t})$, which can be calculated using equation \eqref{eq: continuityeqn} and \eqref{eq: HamiltonJacobi}.
\begin{lem} \label{lem: explicitderivatiesforV}
    Suppose $\mathcal{V}(\mu) = \int V \d \mu$, for a twice continuously differentiable convex function $V$. Let $\grad \theta_{t}$ be the velocity field associated with the displacement interpolation $ \{ \mu_{t} \}_{t \in [0,1]}$. Then, 
    \begin{equation}
\frac{\d}{\d t} \mathcal{V}(\mu_{t}) = \int \langle \grad V, \grad \theta_{t} \rangle \d \mu_{t}, \, \frac{\d^{2}}{\d t^{2}} \mathcal{V}(\mu_{t}) =  \int \langle \grad^{2}V \cdot \grad \theta_{t}, \grad \theta_{t} \rangle \d \mu_{t}.
    \end{equation}
\end{lem}
\begin{proof}
    The formula for the first derivative is the continuity equation \eqref{eq: continuityeqn} itself. For the second derivative, we compute
    \begin{equation} 
\begin{split}
    \frac{\d^{2}}{\d t^{2}} \mathcal{V}(\mu_{t}) &= \frac{\d^{2}}{\d t^{2}} \int V (T_{t}(x)) \d \mu_{0}  = \frac{\d}{\d t} \int \langle \grad V (T_{t}(x)) , \frac{\d }{\d t} T_{t} (x) \rangle \d \mu_{0} (x)\\
    &= \int \langle \grad^{2}V (T_{t}(x)) \cdot \frac{\d}{\d t}T_{t}(x) , \frac{\d}{\d t}T_{t}(x) \rangle \d \mu_{0} + \int \langle \grad V (T_{t}(x)) , \frac{\d^{2}}{\d t^{2}} T_{t} (x) \rangle \d \mu_{0}  \\
    &= \int   \langle \grad^{2}V (T_{t}(x)) \cdot \grad \theta_{t} (T_{t}(x)) , \grad \theta_{t} (T_{t}(x)) \rangle \d \mu_{0} + 0 \\
    &= \int \langle \grad^{2}V \cdot \grad \theta_{t}, \grad \theta_{t} \rangle \d \mu_{t}, \\
\end{split}
\end{equation}
where $T_{t} = (1-t)I + tT$ denotes the optimal map transporting $\mu_{0}$ to $\mu_{t}$ which is used to define $\grad \theta_{t}$. 
\end{proof}
\begin{rem} \label{rem: formulatrueforallflows}
    From the proof, one can observe that the formula for the second derivative only uses the fact that each particle travels along a straight line with constant speed, that is, $\frac{\d^{2}}{\d t^{2}} T_{t} (x) = 0$. 
\end{rem}
Therefore, to obtain the concavity of $e^{- a \mathcal{V}(\cdot)}$ on a displacement convex subset of $\wasser(\R^n)$, we need to prove the inequality $ \int \langle \grad^{2} V \cdot \grad \theta_{t} , \grad \theta_{t} \rangle \d \mu_{t} \geq a \left( \int \langle \grad V, \grad \theta_{t} \rangle \d \mu_{t} \right)^{2}$ for displacement interpolants $\{\mu_{t} \}$ through that set. The following Cauchy--Schwarz inequality is useful for this purpose:
\begin{lem} \label{lem: twistedholder}
    Suppose $A(t)$ is a positive-definite matrix, for every $t$. Then,
    \begin{equation}
 \left(  \int \langle \grad \theta_{t}, \grad V \rangle \d \mu_{t}  \right)^{2} \leq \left( \int \langle A \cdot \grad \theta_{t} , \grad \theta_{t} \rangle \d \mu_{t}  \right) \left( \int \langle A^{-1} \cdot \grad  V , \grad V \rangle \d \mu_{t}\right).
    \end{equation}
\end{lem}
\begin{proof}
We tread a beaten path. Let $s \in \R$, then
    \begin{equation}
\begin{split}
    &0 \leq \int \vert A^{1/2} \cdot \grad \theta_{t} - s A^{-1/2} \cdot \grad V \vert^{2}_{2} \d \mu_{t} \\
    &= \left( \int \langle A \cdot \grad \theta_{t} , \grad \theta_{t} \rangle \d \mu_{t}  \right) + s^{2} \left( \int \langle A^{-1} \cdot \grad  V , \grad V \rangle \d \mu_{t}\right) - 2s  \left(  \int \langle \grad \theta_{t}, \grad V \rangle \d \mu_{t}  \right). 
\end{split}
    \end{equation}
    Now the observation that the quadratic function $As^2 - 2Bs + C$ achieves its minimal value when $s= B/A$ completes the proof. 
\end{proof}
Using the above lemma with $A = \grad^{2}V$ gives,
\begin{equation} \label{eq: twisterholderinaction}
\left( \frac{\d }{\d t} \mathcal{V}(\mu_{t}) \right)^2  \leq \left( \frac{\d^{2}}{\d t^{2}} \mathcal{V}(\mu_{t}) \right) \left( \int \langle \left(\grad^{2}V\right)^{-1} \cdot \grad  V , \grad V \rangle \d \mu_{t}\right).
\end{equation}
Thus, we can upper bound $ \int \langle \left(\grad^{2}V\right)^{-1} \cdot \grad  V , \grad V \rangle \d \mu_{t}$ by $\frac{1}{a}$ to get the concavity of $e^{- a \mathcal{V}(\cdot)}$. When $V$ is \emph{$p$-homogeneous}, that is, when $V( \lambda x) = \lambda^{p} V(x)$ for all $\lambda > 0$, this quantity is a multiple of $V$ itself. In the remainder of this section, all unspecified evaluations should be understood as being evaluated at $x$. For example, $\int \langle \grad V , x \rangle \d \mu$ stands for $ \int \langle \grad V (x) , x \rangle \d \mu (x)$.
\begin{lem} \label{lem: the thing when V is p homo}
    Let $V$ be a $p$-homogeneous convex function. Then, 
    \begin{equation}
    \langle \left(\grad^{2}V\right)^{-1} \cdot \grad  V , \grad V \rangle = \frac{1}{p-1} \langle \grad V , x \rangle = \frac{p}{p-1}V.
    \end{equation}
\end{lem}
\begin{proof}
    If $V$ is $p$-homogeneous, then $\langle \grad V , x \rangle = \lim_{t \to 0} 
\frac{V(x + tx) - V(x)}{t} = p V$, implying the second equality. Differentiating this equation we get $\grad^{2}V \cdot x = (p-1) \grad V$, which implies the first equality in the assertion. 
\end{proof}
We now bound $\int \langle \grad V , x \rangle \d \mu$ for a suitable class of measures (and general $V$), which will give us enough ingredients to put together a displacement concavity statement when $V$ is $p$-homogenous. The main argument of \cite[Lemma 5.3]{KolesnikovLivshyts21} extends to our slightly more general case.  Recall that a function $f : \R^n \to \R$ is said to be \emph{radially decreasing} if $f(\lambda x) \leq f(x)$, for all $x \in \R^n, \lambda \geq 1$. 

\begin{lem} \label{lem: for radially decreasing}
    Suppose $\mu \in \wasser(\R^n)$ is absolutely continuous with respect to the measure $\nu = e^{-V} \d x$, $V$ convex. If the density $\frac{\d \mu }{\d \nu}$ is radially decreasing, then $\int \langle \grad V, x \rangle \d \mu \leq n$.
\end{lem}
\begin{proof}
    Write $f = \frac{\d \mu }{\d \nu}$, so that $\frac{d \mu}{\d x} = f e^{-V}$. By approximation, we may assume that both $f$ and $V$ are smooth. Now, the function
    \begin{equation}
F (s) = \int f(sx) \d \nu (x) = \int f(sx) e^{-V(x)} \d x, 
    \end{equation}
    is clearly decreasing in $s$. Therefore,
    \begin{equation}
    \begin{split}
0 &\geq F'(1) = \int \langle \grad f(x) , x \rangle e^{-V(x)} \d x = \int \langle \grad f(x) , x e^{-V(x)} \rangle \d x \\
&= - \int f(x) \grad \cdot (x e^{-V(x)}) \d x 
= - \int f(x) \left( n e^{-V(x)} - \langle x , \grad V(x) \rangle e^{-V(x)} \right) \d x \\
&=  \int \langle \grad V , x \rangle \d \mu - n. \qedhere
\end{split} 
    \end{equation}
\end{proof}
This sets us up to obtain the following displacement concavity result. 
\begin{thm} \label{thm: dimdispconV}
    Let $V: \R^n \to [0 , \infty)$ be a $p$-homogeneous convex function, $p \in (1, \infty)$. Let $\nu$ be a probability measure of the form $\d \nu = e^{-V + c} \d x$, for some normalizing constant $c$. Then, the following statements are true. 
    \begin{enumerate}
    \item The set $S = \{ \mu : \mathcal{V}(\mu) \leq \frac{n}{p} \}$ is displacement convex and contains all measures whose density with respect to $\nu$ is radially decreasing.
    \item The functional  $e^{- \frac{p-1}{n}\mathcal{V}(\cdot)}$, where $\mathcal{V}(\mu) = \int V \d \mu$, is displacement concave on $S$.   
    \end{enumerate}
\end{thm}
\begin{proof}
First, note that $S$ is a displacement convex set because $\mathcal{V}$ is displacement convex on all of $\wasser(\R^n)$ (which is a direct consequence of the convexity of $V$, see \cite[Theorem 5.15 (ii)]{Villani03}). Moreover, $S = \{ \mu :  \int \langle \grad V , x \rangle \d \mu \leq n \}$, so by Lemma \ref{lem: for radially decreasing} it contains all measures with radially decreasing density with respect to $\nu$.

    For the displacement concavity statement, we begin by assuming that $V$ is twice continuously differentiable. Recall that the derivative criteria Lemma \ref{lem: whatitmeansinderivative} and the Cauchy--Schwarz applied to our case (Equation \eqref{eq: twisterholderinaction}) imply that we need to show
    \begin{equation}
\int \langle \left(\grad^{2}V\right)^{-1} \cdot \grad  V , \grad V \rangle \d \mu \leq \frac{n}{p-1},
    \end{equation}
    for all $\mu \in S$. However, this is precisely the defining property of $S$ in light of Lemma \ref{lem: the thing when V is p homo}.

    Finally we relax the smoothness assumption on $V$ using the following argument. Since $V$ is $p$-homogeneous and convex, $V^{1/p}$ is $1$-homogeneous and quasi-convex. Now every $1$-homogeneous quasi-convex function is necessarily convex, thus $V = h_{K}^{p}$ for some convex body $K$, where $h_{K}(x) = \sup_{y \in \mathbb{S}^{n-1}} \langle x , y \rangle$ denotes the support function of $K$. Recall that $K$ can be approximated by an increasing sequence of convex bodies $K_{i}$, with smooth support functions, which converge to $K$ in the Hausdorff metric (see, \cite[Section 3.4]{Schneider14}). Thus, we can obtain an increasing sequence of smooth $p$-homogeneous convex functions $V_{i} = h_{K_i}^p$ converging to $V$ monotonically. By dominated convergence, $\mathcal{V}_i(\mu) \to \mathcal{V}(\mu)$ for all $\mu$.  Each $S_i = \{ \mu : \mathcal{V}_{i}(\mu):= \int V_{i} \d \mu \leq \frac{n}{p} \}$ is a displacement convex set, therefore $S = \bigcap_{i=1}^\infty S_i$ is also displacement convex. Moreover, the functionals $e^{- \frac{p-1}{n}\mathcal{V}_{i}(\cdot)}$ are displacement concave on $S$, so this is also true for their pointwise limit 
    $e^{- \frac{p-1}{n}\mathcal{V}(\cdot)}$. The proof concludes by taking limits.
\end{proof}
Now that we know $e^{\frac{h(\cdot)}{n}}$ and $e^{- \frac{p-1}{n}\mathcal{V}(\cdot)}$ are both displacement concave on $S$, Theorem \ref{thm:main-entropy}, restated below, is a straightforward consequence. 
\begin{thm} \label{thm: dimdispconD}
    Let $V: \R^n \to [0 , \infty)$ be a $p$-homogeneous convex function, $p \in (1, \infty)$. Let $\nu$ be a probability measure of the form $\d \nu = e^{-V + c} \d x$, for some normalizing constant $c$. Then, the functional $e^{- \frac{p-1}{pn}D(\cdot \Vert \nu)}$ is displacement concave on the set $S = \{ \mu : \int V \d \mu \leq \frac{n}{p} \}$, which contains all measures with radially decreasing density with respect to $\nu$.
\end{thm}
\begin{proof}
    Observe that, 
    \begin{equation}
e^{- \frac{p-1}{pn}D(\cdot \Vert \nu)} = e^{\frac{p-1}{pn}c} \left( e^{\frac{h(\cdot)}{n}} \right)^{1- \frac{1}{p}} \left( e^{- \frac{p-1}{n}\mathcal{V}(\cdot)} \right)^{\frac{1}{p}},
    \end{equation}
    where $\mathcal{V}(\mu) = \int V \d \mu$. Now, since $(x,y) \mapsto x^{1 - \frac{1}{p}}y^{\frac{1}{p}}$ is concave on $[0, \infty)^{2}$ and increasing in each coordinate, it suffices to prove the displacement concavity of $e^{\frac{h(\cdot)}{n}}$ and $e^{- \frac{p-1}{n}\mathcal{V}(\cdot)}$ separately. Since this has already been achieved in Theorem \ref{thm: EKS15} and Theorem \ref{thm: dimdispconV} for $S$, the asserted displacement concavity result for $S$ follows. 
\end{proof}
Finally, observe that the density of $\nu_{K}$ with respect to $\nu$, if $K$ is a star body, is a radially decreasing function. Therefore, Theorem \ref{thm: dimdispconD} above and Proposition \ref{pro: wand}, together imply Theorem \ref{thm:main-BM intro} restated below.

\begin{thm}
\label{thm:main-BM}Let $V:\RR^{n}\to[0,\infty)$ be a $p$-homogeneous
convex function for $1<p<\infty$, and let $\nu$ be a measure with
density $\frac{\dd\nu}{\dd x}$ proportional to $e^{-V}$. Let $A,B\subseteq\RR^{n}$
be star bodies. Then for all $0\le\lambda\le1$ we have
\begin{equation}
\text{\ensuremath{\nu\left((1-\lambda)A+\lambda B\right)^{\frac{p-1}{pn}}\ge(1-\lambda)\nu(A)^{\frac{p-1}{pn}}+\lambda\nu\left(B\right)^{\frac{p-1}{pn}}.}}
\end{equation}
\end{thm}
\qed

\section{Dimensional convexity properties of the Gaussian relative entropy} \label{sec: gaussian}
In this section, we will differentiate the relative entropy $D(\cdot \Vert \nu)$, where $\d \nu = e^{-V} \d x$, directly along displacement interpolations without decomposing it into two parts as in the previous section. 
While we focus on the Gaussian measure, $\nu = \gamma$, we will introduce the relevant setup for a general $\nu$.

In this direction, the presentation is cleaner and more insightful if we use the language of the Bakry--\'Emery $\Gamma$-calculus. Measures $\d \nu = e^{-V} \d x$, for smooth enough $V$, naturally occur as the equilibrium measure for the semigroup of operators $P_{t} = e^{tL}$, with generator $Lf = \Delta f - \langle \grad V , \grad f \rangle$. Introduction of the bilinear form $\Gamma$ (called the carr\'e du champ) on functions, and its iteration $\Gamma_{2}$, reduces the study of many properties regarding convergence to equilibrium to the study of the relationships among $L$, $\Gamma$, and $\Gamma_{2}$. 

In the case of $\d \nu = e^{-V} \d x$ under discussion, when the generator $L$ is as described earlier, the carr\'e du champ and its iteration are given by,
\begin{equation}
\begin{split}
    & \Gamma (f,g) = \langle \grad f , \grad g \rangle, \\
    & \Gamma_{2} (f,g) = \langle \grad^{2} f , \grad^{2} g \rangle + \langle \grad^{2}V \cdot \grad f , \grad g \rangle,
\end{split}
\end{equation}
where the first inner-product in the expression for $\Gamma_{2}$ is the Hilbert--Schmidt inner product. It is customary to write $\Gamma (f)$ for $\Gamma (f, f)$, and $\Gamma_{2} (f)$ for $\Gamma_{2}(f,f)$. We refer the reader to the text \cite{BakryGentilLedoux14} for a comprehensive treatment of this subject. However, what we will use in calculating the derivatives of $D(\mu_{t} \Vert \nu)$ is a simple integration-by-parts formula that the reader may verify directly,
\begin{equation} \label{eq: integrationbyparts}
 \int \Gamma (f,g) \d \nu = - \int f L g \d \nu.
\end{equation}
\begin{exmpl}
    When $V = 0$, one obtains the \emph{heat semigroup} which has the Lebesgue measure as its equilibrium measure. In this case, $L f= \Delta f$, $\Gamma (f) = \vert \grad f \vert^{2}$, and $\Gamma_{2} (f) = \Vert \grad^{2} f \Vert^{2}$. Here $\Vert \cdot \Vert$ denotes the Hilbert--Schmidt norm. 
\end{exmpl}
\begin{exmpl} \label{exmpl: OU}
    When $V$ corresponds to the Gaussian measure on $\R^n$, $\nu = \gamma$, the semigroup obtained is the well-known \emph{Ornstein--Uhlenbeck semigroup}. Then $Lf = \Delta f - \langle x, \grad f \rangle$, $\Gamma (f) = \vert \grad f \vert^{2}$, and $\Gamma_{2} (f) = \Vert \grad^{2} f \Vert^{2} + \vert \grad f \vert^{2}$. 
\end{exmpl}
Let us differentiate $D(\cdot \Vert \nu)$ two times, as needed. The underlying heuristic, namely that $ P_{t}  f \d \nu$ is the gradient flow of $D( \cdot \Vert \nu)$ (starting at $f \d \nu$) under a natural Riemannian-like structure on $\wasser (\R^n)$ introduced by Otto (\cite{Otto01}), is a powerful tool in the field. However, we will be content with the following simple calculation and refer the reader to \cite[Chapter 8]{Villani03} or \cite[Chapter 15]{Villani09} for more details.
\begin{lem} \label{lem: derivativesofentropy}
    Let $\d \nu = e^{-V} \d x$, for smooth $V$. Consider a displacement interpolation $\{ \mu_{t} \}_{t \in [0,1]}$ in $\wasser (\R^n)$, with the optimal velocity field $\grad \theta_{t}$. Suppose $\theta_{t}$ is sufficiently smooth, then 
    \begin{equation}
\begin{split}
    &\frac{\d}{\d t} D(\mu_{t} \Vert \nu) = - \int L \theta_{t} \d \mu_{t}, \\
    &\frac{\d^{2}}{\d t^{2}}  D(\mu_{t} \Vert \nu) = \int \Gamma_{2} (\theta_{t}) \d \mu_{t}., \\
\end{split}
    \end{equation}
    where $L, \Gamma_{2}$ are as described earlier. 
\end{lem}
\begin{proof}
    Let $ \rho_{t}$ denote the density of $\mu_{t}$ with respect to $\nu$. Then,
    \begin{equation} \label{eq: firstderivativeentropycomp}
\begin{split}
\frac{\d}{\d t} D(\mu_{t} \Vert \nu) &= \frac{\d}{\d t} \int \rho_{t} \log \rho_{t} \d \nu = \frac{\d}{\d t} \int  \log \rho_{t} \d \mu_{t} \\
&= \int \partial_{t}  \log \rho_{t} \d \mu_{t} + \int \langle \grad \log \rho_{t} , \grad \theta_{t} \rangle \d \mu_{t} \\
&=  \frac{\d}{\d t}\int  \rho_{t} \d \nu + \int \langle \grad \log \rho_{t} , \grad \theta_{t} \rangle \rho_{t} \d \nu\\
&= 0 + \int \langle \grad \rho_{t} , \grad \theta_{t} \rangle \d \nu = - \int \rho_{t} \, L \theta_{t} \d \nu = - \int L \theta_{t} \d \mu_{t}. 
\end{split}
\end{equation}
In the above computation, the third equality uses the chain rule and the continuity equation \eqref{eq: continuityeqn}, and the sixth equality is an application of formula \eqref{eq: integrationbyparts} for the semigroup under consideration.

Furthermore, we can differentiate again to compute the second derivative, 
\begin{equation}
\begin{split}
\frac{\d^{2}}{\d t ^{2}}  D(\mu_{t} \Vert \nu) &= -  \frac{\d }{\d t} \int  L \theta_{t} \d \mu_{t} = - \int  L \del_{t} \theta_{t} \d \mu_{t} - \int \langle \grad L \theta_{t} , \grad \theta_{t} \rangle  \d \mu_{t}  \\
&= \int  L \left( \frac{|\grad \theta_{t} |^{2}}{2} \right)  \d \mu_{t} - \int \langle \grad L \theta_{t} , \grad \theta_{t} \rangle  \d \mu_{t} \\
&= \int \left(  L \left( \frac{|\grad \theta_{t} |^{2}}{2} \right) -   \langle \grad L \theta_{t} , \grad \theta_{t} \rangle \right) \d \mu_{t} = \int \Gamma_{2} (\theta_{t}) \d \mu_{t},
\end{split}
\end{equation}
where in the second equality we use the change rule and the continuity equation \eqref{eq: continuityeqn}, the third equality uses the Hamilton--Jacobi equation \eqref{eq: HamiltonJacobi}, while the last equality uses the formula
\begin{equation}
\Gamma_{2} (f) =  L \left( \frac{|\grad f |^{2}}{2} \right) -   \langle \grad L f , \grad f \rangle 
\end{equation}
for smooth $f$. This is a well-known weighted version of Bochner's formula, which can be easily obtained from its unweighted counterpart (see, for example, \cite[Chapter 14]{Villani09}).
\end{proof}
Therefore, if $S \subseteq \wasser (R^{n})$ is a displacement convex set, and we would like to prove that $e^{- a D (\cdot \Vert \nu)}$ is concave on $S$, we must prove that
\begin{equation}
\int \Gamma_{2} (\theta_{t}) \d \mu_{t} \geq a \left( \int L \theta_{t} \d \mu_{t} \right)^{2},
\end{equation}
for all displacement interpolations $\{ \mu_{t} \}_{t \in [0,1]}$ in $S$. Actually, in this section, we would like to go a little bit further than the above inequality to also utilize the ``curvature'' associated with the semigroup. To this end, we introduce the notion  of $(\kappa, 1/a)$-convexity of a function defined on a uniquely geodesic space. 

\begin{definition} \cite[Equation 1.2 rephrased]{ErbarKuwadaSturm15} \label{def: knconvexity}
    Let $(\mathcal{X},d)$ be a uniquely geodesic metric space. A function $S: \mathcal{X} \to [0, \infty]$ is called \emph{$(\kappa,1/a)$-convex} if for every constant speed geodesic $c : [0,1] \to \mathcal{X}$, we have
    \begin{equation}
 \frac{\d^{2}}{\d t^{2}} S(c_{t}) \geq \kappa d(c_{0}, c_{1})^{2} + a \left( \frac{\d }{\d t}S(c_{t}) \right)^{2},
    \end{equation}
    in the distributional sense. 
\end{definition}
\begin{rem}
    A global formulation of the above condition can be found in \cite[Definition 2.7]{ErbarKuwadaSturm15}. We discuss the particular case of $(2,n)$-convexity in Lemma \ref{lem: 2ncon}.
\end{rem}
Thus, Theorem \ref{thm: dimdispconD} from previous section establishes $(0, \frac{p}{p-1}n)$-convexity of the relative entropy $D(\cdot \Vert \nu)$ on the geodesic subspace $S = \{ \mu: \int V \d \mu \leq \frac{n}{p} \} \subseteq \wasser(\R^n)$ when $\d \nu = e^{-V} \d x$, where $V$ is a $p$-homogeneous convex function on $\R^n$.  In this case, it is possible to trade some of the dependence on $n$ for a dependence on the lower bounds on $\grad^{2} V$, but we choose not to pursue this direction in the current paper.

We will now specialize to the Gaussian case. Note that this corresponds to the semigroup in Example \ref{exmpl: OU}. Here we aim to obtain a $(\kappa, n)$-convexity of $D(\cdot \Vert \gamma)$ on a suitable class of measures. Thus, we want to obtain an inequality that reads,
\begin{equation} \label{eq: entropiccdkn}
\int \Gamma_{2}(\theta_{t}) \d \mu_{t} \geq \kappa \wasser^{2}(\mu_{0}, \mu_{1}) + \frac{1}{n} \left( \int L \theta_{t} \d \mu_{t} \right)^{2},
\end{equation}
that is,
\begin{equation}
\int \Vert \grad^{2} \theta_{t} \Vert^{2} + \vert \grad \theta_{t} \vert^{2} \d \mu_{t} \geq  \kappa \wasser^{2}(\mu_{0}, \mu_{1}) + \frac{1}{n} \left( \int L \theta_{t} \d \mu_{t} \right)^{2}.
\end{equation}
This is remarkably similar to the sufficiency condition in \cite{KolesnikovLivshyts21} obtained from the Kolesnikov--Milman approach (\cite{KolesnikovMilman18}) with the exception of the first term on the right. This comparison goes beyond similarity, to the level of the proof itself, since it is the same as in \cite{EskenazisMoschidis21} albeit presented in a slightly different form. While we claim no originality here, for the sake of completeness we repeat the argument from \cite{EskenazisMoschidis21} adapted to obtain the statement we want. 

\begin{thm} \label{thm: 2nconunderconditions}
Denote by $S_{n}$ the collection of all even measures $\mu$ in $\mathcal{P}_{2,ac}(\R^n)$ which satisfy the inequality $\int f^{2} \d \mu \leq \int \vert \grad f \vert^{2} \d \mu$ for all odd test functions $f$. Then, 
\begin{equation}
\int \Gamma_{2} (u) \d \mu \geq 2 \int \Gamma (u) \d \mu + \frac{1}{n} \left( \int L u \d \mu \right)^{2},
\end{equation}
for all even functions $u$. Here $L$, $\Gamma_{2}$ are the generator and the iterated carr\'e du champ, respectively, for the Ornstein--Uhlenbeck semigroup. 

Consequently, if $\{ \mu_{t} \}_{t \in [0,1]}$ is a displacement interpolation that lies entirely in $S_{n}$ and is driven by a smooth velocity field, then $D(\mu_{t}\Vert \gamma)$ is $(2,n)$-convex. In particular,
\begin{equation} \label{eq: dispcongaussiannok}
e^{-\frac{D (\mu_{t} \Vert \gamma)}{n}} \geq (1-t) e^{-\frac{D (\mu_{0} \Vert \gamma)}{n}} + t e^{-\frac{D (\mu_{1} \Vert \gamma)}{n}}.
\end{equation}
\end{thm}
\begin{rem}
    If the measure $\mu$ satisfies the inequality $\var_{\mu}(f) \leq C \int \vert \grad f \vert^{2} \d \mu$, we say that $\mu$ satisfies a \emph{Poincar\'e inequality} with constant $C$. The best possible $C$ is called the \emph{Poincar\'e constant} of $\mu$. If we assume that the function $f$ is odd and the measure $\mu$ is even, then $\var_{\mu}(f) = \int f^{2} \d \mu$. Thus, $S_{n}$ is the collection of even measures, absolutely continuous with respect to the Lebesgue measure and having finite second moment, which satisfy a Poincar\'e inequality for odd functions with constant $1$. 
\end{rem}
\begin{proof}
     Let $u$ be an even function. Consider $v = u  -  r$, where $r(x) = \frac{\int L u \d \mu}{2n} |x|^{2}$. We will write $l = \int Lu \d \mu$ for notational convenience. Note that $\grad r (x) = \frac{l}{n}x, \grad^{2} r (x) = \frac{l}{n} I$, where $I$ denotes the identity matrix, and $\Delta r = l$.  
Therefore, 
\begin{equation}
\begin{split}
\Vert \grad^2 u \Vert^{2} 
&= \Vert \grad^2 v + \grad^{2} r \Vert^{2} 
= \Vert \grad^{2} v \Vert^{2} + \frac{2l}{n} \Delta v + \frac{l^2}{n} 
= \Vert \grad^2 v \Vert^{2} + \frac{2 l}{n} \Delta u - \frac{l^2}{n} \\
&=  \Vert \grad^2 v \Vert^{2} + \frac{2l}{n} \left(   L u + \< x , \grad u\>  \right) - \frac{l^2}{n} \\
&= \Vert \grad^2 v \Vert^{2} + \frac{2l}{n}  \< x , \grad u\> + \left( \frac{2l}{n} Lu - \frac{l^2}{n} \right).
\end{split}
\end{equation}
Now, since $v$ is even, $\partial_{i}v$ is an odd function, for each $i$. Therefore, if $\mu \in S_{n}$ then we can apply the Poincar\'e inequality assumption to $\partial_{i} v$ in the calculation below. 
\begin{equation}
\begin{split}
 \int \Vert \grad^2 v \Vert^2 \d \mu 
 &= \sum_{i} \int \sum_{j} \left( \partial_{j} \partial_{i} v \right)^{2} \d \mu 
 = \sum_{i} \int \vert \grad \partial_{i} v \vert^{2} \d \mu  
  \geq \sum_{i} \int \left( \partial_{i} v \right)^2 \d \mu \\
  &= \sum_{i} \int \left(  \partial_{i} u_{i} - \frac{l}{n}x_{i} \right)^{2} \d \mu
  = \int \left( | \grad u |^{2} - \frac{2l}{n} \< x , \grad u \> + \frac{l^2}{n^{2}} |x|^{2} \right) \d \mu.
 \end{split}
\end{equation}

Putting this inequality into the expression for $\Vert \grad^{2} u \Vert^{2}$ obtained earlier,
\begin{equation}
\begin{split}
 \int  \Gamma_{2} (u) \d \mu &= \int \left( \Vert \grad^{2} u \Vert^{2} + \vert \grad u \vert^{2} \right) \d \mu  \\
 &= \int \left( \Vert \grad^2 v \Vert^{2} + \frac{2l}{n}  \< x , \grad u\> + \left( \frac{2l}{n} Lu - \frac{l^2}{n} \right) + \vert \grad u \vert^{2} \right) \d \mu \\
 &\geq \int \left( | \grad u |^{2} - \frac{2l}{n} \< x , \grad u \> + \frac{l^2}{n^{2}} |x|^{2}  + \frac{2l}{n}  \< x , \grad u\> + \left( \frac{2l}{n} Lu - \frac{l^2}{n} \right) + \vert \grad u \vert^{2} \right) \d \mu  \\
 &= \int \left( 2 | \grad u |^{2}  + \frac{l^{2}}{n^{2}} |x|^{2} +  \left( \frac{2l}{n} Lu - \frac{l^2}{n} \right) \right) \d \mu \\
 &\geq \int \left( \frac{2l}{n} Lu - \frac{l^2}{n} \right) \d \mu + 2 \int \vert \grad u \vert^{2} \d \mu = \frac{1}{n} l^{2} +  2 \int \vert \grad u \vert^{2} \d \mu  \\
 &=\frac{1}{n} \left( \int L u \d \mu \right)^{2} + 2 \int \Gamma (u) \d \mu.
\end{split}
\end{equation}
For the second part, observe that a potential $\theta_{t}$ for the velocity field $\grad \theta_{t}$ corresponding to optimal transport between any two measures $\mu_{0}$ and $\mu_{1}$ in $S_{n}$ must be even, since the measures themselves are even. Further, $\int \Gamma (\theta_{t}) \d \mu_{t} = \int \vert \grad \theta_{t} \vert^{2} \d \mu_{t}$, but the latter quantity equals $\wasser^{2}(\mu_{0}, \mu_{1})$ by \cite[Remarks 5.40 (iii), Example 5.42]{Villani03}. Therefore, by the first part, if the displacement interpolation between $\mu_{0}$ and $\mu_{1}$ is contained in $S_{n}$ then the condition of inequality \eqref{eq: entropiccdkn} is satisfied. This establishes the second part of the theorem.
\end{proof}
\begin{rem}
    As indicated in the Introduction, an affirmative answer to Question \ref{qstn: poincare} for $\mu_{0} = \gamma_{K_{0}}, \mu_{1} = \gamma_{K_{1}}$, for symmetric convex bodies $K_{0}$ and $K_{1}$, would imply the result of Eskenazis--Moschidis (\cite{EskenazisMoschidis21}).
\end{rem}
\begin{definition}
    A measure $\mu \in \mathcal{P}(\R^n)$ is said to be \emph{strongly log-concave} if its density with respect to the Gaussian measure $\gamma$ on $\R^n$ exists, and is a log-concave function.
\end{definition}
\begin{exmpl}
    Let $K$ be a convex body. Then, $\gamma_{K}$ is strongly log-concave because $\frac{\d \gamma_{K}}{\d \gamma} = \frac{1}{\gamma(K)} \ind_{K}$ is a log-concave function. 
\end{exmpl}

Returning to our investigation of displacement convexity properties of $S_{n}$, we note that $S_{n}$ itself is displacement convex when $n=1$.
\begin{thm}
    The set $S_{1} \subseteq \wasser (\R)$ is displacement convex. Therefore, in particular,
    \begin{equation}
e^{-D(\mu_{t} \Vert \gamma)} \geq (1-t) e^{-D(\mu_{0} \Vert \gamma)} + t e^{- D(\mu_{1} \Vert \gamma)},
    \end{equation}
    when $\mu_{0}, \mu_{1}$ are even strongly log-concave probability measures. 
\end{thm}
\begin{proof}
    Let $\mu_{0}, \mu_{1} \in S_{1}$. Denote by $T$ the optimal map pushing forward $\mu_{0}$ to $\mu_{1}$ guaranteed by Theorem \ref{thm: Brenier}. Thus, $(T_{t})_{\#}\mu_{0} = \mu_{t}$, $t \in [0,1]$, is the displacement interpolation. Let $f$ be a continuously-differentiable odd function supported within $\supp (\mu_{0})$. The main idea is to choose functions $f_{t}$ such that $f'_{t} = f' \circ T_{t}^{-1}$. We let 
    \begin{equation}
f_{t} (x) = \int_{0}^{x} f' (T_{t}^{-1}(y)) \d y.
    \end{equation}
    Therefore, 
    \begin{equation}
\begin{split}
\int \vert f_{t} \vert^{2} \d \mu_{t} 
&= \int \vert f_{t} (T_{t} (x)) \vert^{2} \d \mu_{0} (x) = \int \left( \int_{0}^{T_{t}(x)} f' (T_{t}^{-1}(y)) \d y \right)^{2} \d \mu_{0} (x) \\
&= \int \left(  \int_{0}^{x} f'(z)T'_{t}(z) \d z   \right)^{2} \d \mu_{0}(x) = \int \left( (1-t) \int_{0}^{x} f' + t \int_{0}^{x} f'T' \right)^{2} \d \mu_{0}, \\
\end{split}
    \end{equation}
    which is clearly convex in $t$.
    Thus, 
    \begin{equation}
    \begin{split}
\int \vert  f_{t} \vert^{2} \d \mu_{t} &\leq (1-t) \int \vert f_{0} \vert^{2} \d \mu_{0} + t \int \vert f_{1} \vert^{2} \d \mu_{1}\\
& \leq (1-t) \int \vert f'_{0} \vert^{2} \d \mu_{0} + t \int \vert f'_{1} \vert^{2} \d \mu_{1} = \int \vert f'_{t} \vert^{2} \d \mu_{t},
\end{split}
    \end{equation}
   where we have used the fact that the integral
    \begin{equation}
    \int \vert f'_{t} \vert^{2} \d \mu_{t} = \int \vert f' \circ T_{t}^{-1} \vert^{2} \d \mu_{t} = \int \vert f' \vert^{2} \d \mu_{0}
    \end{equation}
    remains constant in time $t$. The one-to-one correspondence obtained by the interpolation procedure $f_{t}$ among the test functions supported within each $\supp(\mu_{t})$ then concludes the proof of the first assertion. Given the second assertion is about optimal transport between log-concave measures, the regularity of the velocity required to draw the desired conclusion is furnished by Caffarelli's regularity theory (\cite{Caffarelli90, Caffarelli92}) and approximation. 
\end{proof} 
For higher dimensions, the fact that not all vector fields are conservative, prevents us from finding the function $f_{t}$ satisfying $\grad f_{t} = \grad f \circ T_{t}^{-1}$. Yet, we can derive some partial results.
\begin{prop} \label{prop: oneendgaussian}
The following displacement convex subsets of $\wasser (\R^n)$ lie completely in $S_{n}$.  
    \begin{enumerate}
        \item The collection of all centered Gaussians whose covariance matrix is dominated by the identity,
        $
\{ ( \Sigma^{1/2} )_{\#} \gamma : 0 < \Sigma \leq I_{n \times n} \} =: S_{n,\mathcal{G}}.
        $
        \item Geodesic segments $\{ \mu_{t} : t \in [0,1] \}$, when one endpoint is an even strongly log-concave probability measure and the other is $\gamma$ itself. 
     \end{enumerate}   
     Moreover, inequality \eqref{eq: dispcongaussiannok} holds whenever both endpoints lie in any one of these sets. 
\end{prop}
\begin{proof}
For the first set, use the following facts which are consequences of the Brascamp--Lieb inequality and Theorem \ref{thm: Brenier}.
    \begin{itemize}
        \item Affine functions are the extremal functions in the Poincar\'e inequality for Gaussian measures. 
        \item The collection of the Gaussians considered in the proposition satisfy the Poincar\'e inequality with constant $1$.
        \item The collection of all centered Gaussian measures (that is, all linear images of the standard Gaussian $\gamma$) is a displacement convex set. 
    \end{itemize}
We basically prove the convexity of the Poincar\'e constant in this case. Let $X_{0}, X_{1}$ be random vectors with distributions in the set $S_{n, \mathcal{G}}$. Then their displacement interpolation is represented by a Gaussian random vector $X_{t} = (1-t)X_{0} + t X_{1}$, where $(X_{0},X_{1})$ are optimally coupled.  Suppose $f(x) = \langle x, u \rangle$, for some $u \in \R^n$. Then, 
\begin{equation}
\E f(X_{t})^{2} = \E  \langle u , X_{t} \rangle^{2} \leq (1-t) \E \langle u , X_{0} \rangle^{2} + t  \E \langle u , X_{1} \rangle^{2} \leq \vert u \vert^{2} = \E \vert \grad f (X_{t}) \vert^{2}.
\end{equation}
This completes the proof of the displacement convexity of the set $S_{n, \mathcal{G}}$. The optimal transport map between Gaussian measures is linear (\cite{KnottSmith84}), therefore the velocity field has the regularity needed to deduce the inequality \eqref{eq: dispcongaussiannok} when both endpoints lie in this set. 

For the second set, we will use Caffarelli's contraction theorem (\cite{Caffarelli00}), which says that the optimal map $T$ from $\mu_{0} = \gamma$ to a strongly log-concave measure $\mu_{1}$ is $1$-Lipschitz. Therefore, the $T_{t}$ from Theorem \ref{thm: Brenier} are also $1$-Lipschitz. A standard change of variables argument shows that the Poincar\'e constant of a $1$-Lipschitz image of $\gamma$ can be at most $1$. This shows that all the $\mu_{t}$ in the second part of the proposition have Poincar\'e constants $\leq 1$. We use (approximation and) Caffarelli's regularity theory (\cite{Caffarelli90,Caffarelli92}) again to obtain the required regularity of velocity fields to deduce inequality \eqref{eq: dispcongaussiannok}.

\end{proof}
\section{Applications to functional inequalities} \label{sec: gaussianapp}
In this section, we present some applications of the second part of Proposition \ref{prop: oneendgaussian}, that is, the $(2,n)$-convexity of $D(\mu_{t} \Vert \gamma)$ when one endpoint of the geodesic segment $\{ \mu_{t} \}_{t \in [0,1]}$ is an even strongly log-concave probability measure, and the other is the $\gamma$ itself. More specifically, we will prove improved $(2,n)$ versions to the (a priori $(1, \infty)-$) Gaussian \emph{HWI}, Logarithmic Sobolev, and Talagrand inequalities, for even strongly log-concave probability measures. 

\begin{definition}
   The (relative) \emph{Fisher information} of a measure $\mu$ is defined by,
\begin{equation}
I(\mu \Vert \gamma) = \int \frac{\vert \grad f \vert^{2}}{f} \d \gamma,
\end{equation}
where $f = \frac{\d \mu}{\d \gamma}$.
\end{definition}
As discussed in the Introduction, HWI inequalities relate the three quantities that its initials correspond to. The main idea is to do calculus on the relative entropy $D(\mu_{t} \Vert \gamma)$ along the displacement interpolation between $\mu$ and $\gamma$, using the convexity properties of relative entropy. This accounts for the ``H'' and the ``W''. The Fisher information $I(\mu \Vert \gamma)$ appears  because it plays the role of the squared norm of the gradient of $D(\cdot \Vert \gamma)$ under the previously mentioned Riemannian-like structure on $\wasser(\R^n)$. This makes the Fisher information a natural candidate to bound the derivative of relative entropy along any given (unit) tangent. To the interested reader who is unfamiliar with the idea outlined in this paragraph, we recommend the clear heuristic treatment of the $(K,\infty)$-case in \cite[Section 3]{OttoVillani00}. Logarithmic Sobolev and Talagrand inequalities (which relate $H,I$, and $W,I$, respectively) follow from a suitable HWI inequality. 

We follow the footsteps of \cite[Section 3.4]{ErbarKuwadaSturm15} where the authors derive $(\kappa, n)$-versions of these inequalities from $(\kappa,n)$-convexity of relative entropy in a very general setup. While a global $(\kappa,n)$-convexity condition is assumed in \cite{ErbarKuwadaSturm15}, it is easily seen from the proofs that the requirement is only on the interpolations connecting a measure $\mu$ (for which the inequalities are to be proven) to the reference measure ($\gamma$ in our case). Therefore, for completeness, we will only sketch the main steps in \cite{ErbarKuwadaSturm15} adapted to our setup. The reader is referred to \cite{ErbarKuwadaSturm15} for more details.

To begin with, we decipher what the local condition of $(2,n)$-convexity means globally. Towards this, we need some definitions. 
\begin{definition}\cite[Definition 2.1]{ErbarKuwadaSturm15}
  For $\theta \geq 0$, we define the following functions.
  \begin{equation}
\begin{split}
& \mathfrak{s} (\theta) = \mathfrak{s}_{2/n} (\theta) = \sqrt{\frac{n}{2}} \sin \left( \sqrt{\frac{2}{n}}  \theta \right), \\
& \cc (\theta) = \cc_{2/n} (\theta) = \cos \left(  \sqrt{\frac{2}{n}}  \theta  \right).\\
\end{split}
  \end{equation}
  For $t \in [0,1]$, we set
  \begin{equation}
\sigma^{(t)}(\theta) = \sigma_{2/n}^{(t)}(\theta) = 
\begin{cases}
    \frac{\s (t \theta)}{\s (\theta)}, & \textnormal{if } \theta < \sqrt{\frac{n}{2}} \pi, \\
    \infty, &\textnormal{ otherwise.}\\
\end{cases}
  \end{equation}
\end{definition}
In terms of the coefficients $\sigma^{(t)}$ we may write a global condition for $(2, n)$-convexity.
\begin{lem}\cite[Lemma 2.2]{ErbarKuwadaSturm15} \label{lem: 2ncon}
The $(2,n)$-convexity of $D(\mu_{t} \Vert \gamma)$ implies that, 
\begin{equation} \label{eq: 2ncon}
e^{-\frac{D(\mu_{t} \Vert \gamma)}{n}} \geq \sigma^{(1-t)} (\wasser(\mu_{0}, \mu_{1})) e^{-\frac{D(\mu_{0} \Vert \gamma)}{n}} + \sigma^{(t)} (\wasser(\mu_{0}, \mu_{1})) e^{-\frac{D(\mu_{1} \Vert \gamma)}{n}}.
\end{equation}
\end{lem}
\begin{proof}[Proof outline]
    Let $u(t) = e^{-\frac{D(\mu_{t} \Vert \gamma)}{n}}$ let $v(t)$ denote the RHS in the above inequality. The $(2,n)$-convexity of $D(\mu_{t} \Vert \gamma)$ is equivalent to the condition $u''(t) \leq - \frac{2}{n} \wasser^{2}(\mu_{0}, \mu_{1}) u(t)$. On the other hand, $v(t)$ satisfies $v''(t) =  - \frac{2}{n} \wasser^{2}(\mu_{0}, \mu_{1}) v(t)$. The proof concludes by a maximum principle.
\end{proof}
\begin{rem}
 Lemma \ref{lem: 2ncon} forces $\wasser^{2} (\mu, \gamma) \leq \frac{n}{2} \pi^{2}$ for all even strongly log-concave probability measures because the coefficients $\sigma^{(t)}(\cdot)$ blow up at $\sqrt{\frac{n}{2}} \pi$. However, this is not alarming since the stronger bound  $\wasser^{2} (\mu, \gamma) \leq n$ can be obtained by a straightforward calculation using Caffarelli's contraction theorem and the Gaussian Poincar\'e inequality.
\end{rem}
Using Lemma \ref{lem: 2ncon}, we obtain Theorem \ref{thm: HWIintro}, which is restated below. 
 
\begin{thm}[A $(2,n)$-HWI inequality]
    Let $\mu$ be an even strongly log-concave probability measure on $\R^n$, then,
    \begin{equation} \label{eq: hwi*}
e^{\frac{D(\mu_{0} \Vert \gamma)}{n}- \frac{D(\mu_{1} \Vert \gamma)}{n}}
\leq \cc ( \wasser (\mu_{0}, \mu_{1})) + \frac{1}{n} \s ( \wasser(\mu_{0}, \mu_{1})) \sqrt{I(\mu_{0} \Vert \gamma)},
    \end{equation}
    where $\{ \mu_{0} , \mu_{1} \} = \{ \mu , \gamma \}$.

    In particular,
    \begin{equation} \label{eq: hwi}
e^{\frac{D(\mu \Vert \gamma)}{n}} \leq \cc (\wasser(\mu, \gamma)) + \frac{1}{n} \s(\wasser(\mu, \gamma)) \sqrt{I(\mu \Vert \gamma)}.
    \end{equation}
\end{thm}
\begin{proof}[Proof outline]
    When $\{ \mu_{0}, \mu_{1} \} = \{ \mu , \gamma \}$, the relative entropy $D(\mu_{t} \Vert \gamma)$ is $(2,n)$-convex by Proposition \ref{prop: oneendgaussian} and Theorem \ref{thm: 2nconunderconditions}. Therefore, Equation \eqref{eq: 2ncon} holds. Now subtracting $e^{-\frac{D(\mu_{0} \Vert \gamma)}{n}}$ from both sides of Equation \eqref{eq: 2ncon}, dividing both sides by $t$, and letting $t \to 0$, we obtain
    \begin{equation}
e^{- \frac{D(\mu_{1} \Vert \gamma)}{n}} \leq \cc (\wasser (\mu_{0}, \mu_{1})) \cdot e^{- \frac{D(\mu_{0} \Vert \gamma)}{n}} + \frac{\s (\wasser(\mu_{0}, \mu_{1}))}{\wasser (\mu_{0}, \mu_{1})}  \left. \frac{\d}{\d t} e^{- \frac{D(\mu_{t} \Vert \gamma)}{n}} \right|_{t=0}.
    \end{equation}
    Observing that
    \begin{equation}
    \begin{split}
\left. \frac{\d}{\d t} e^{- \frac{D(\mu_{t} \Vert \gamma)}{n}} \right|_{t=0} &= \frac{1}{n} e^{- \frac{D(\mu_{0} \Vert \gamma)}{n}} \frac{\d }{\d t} \left. \left( - D(\mu_{t} \Vert \gamma) \right) \right|_{t=0} \\
&\leq \frac{1}{n} e^{- \frac{D(\mu_{0} \Vert \gamma)}{n}} \sqrt{I(\mu_{0} \Vert \gamma)} \wasser(\mu_{0}, \mu_{1}) \\
\end{split}
    \end{equation}
    completes the proof of the first inequality. Here we have used a piece from the computation \eqref{eq: firstderivativeentropycomp},
    \begin{equation}
    \begin{split}
\frac{\d }{\d t} \left. \left( - D(\mu_{t} \Vert \gamma) \right) \right|_{t=0} &= \int L \theta_{0} \d \mu_{0} = - \int \langle \grad \log f, \grad \theta_{0} \rangle \d \mu_{0} \\
&\leq \left( \int \vert \grad \log f \vert^{2} \d \mu_{0}  \right)^{1/2} \left( \int \vert \grad \theta_{0} \vert^{2} \d \mu_{0} \right)^{1/2} \\
&= \sqrt{I(\mu_{0} \Vert \gamma)} \wasser(\mu_{0}, \mu_{1}),
\end{split}
    \end{equation}
    where $f = \frac{\d \mu_{0}}{\d \gamma}$ and $\grad \theta_{t}$ denotes the velocity field driving the displacement interpolation from $\mu_{0}$ to $\mu_{1}$.
    
    The second inequality is obtained simply by setting $\mu_{0} = \mu, \mu_{1} = \gamma$.
\end{proof}

As immediate corollaries, dimensional improvements of the Gaussian Logarithmic Sobolev and Gaussian Talagrand inequalities are obtained for even strongly log-concave measures.
 \begin{cor}[A Logarithmic Sobolev inequality] \label{cor: logsobolev}
     Let $\mu$ be an even strongly log-concave probability measure on $\R^n$. Then, 
     \begin{equation} \label{eq: logsobolev}
4 D(\mu \Vert \gamma) \leq 2n \left( e^{\frac{2}{n}D(\mu \Vert \gamma)} - 1 \right) \leq I( \mu \Vert \gamma).
        \end{equation}
 \end{cor}
\begin{proof}
We square both sides of Equation \eqref{eq: hwi} to obtain 
\begin{equation}
e^{\frac{2}{n}D(\mu \Vert \gamma)} \leq \cc(\wasser(\mu, \gamma))^{2} + \frac{2}{n} \s(\wasser(\mu, \gamma)) \cc (\wasser(\mu, \gamma)) \sqrt{I(\mu \Vert \gamma)} + \frac{1}{n^2} \s (\wasser(\mu, \gamma))^{2} I(\mu \vert \gamma).
\end{equation}
The inequality $2 ab \leq 2 a^{2} + \frac{1}{2}b^{2}$ applied to $a = \mathfrak{s} / \sqrt{n}, b = \mathfrak{c} \sqrt{I(\mu \Vert \gamma)/n}$, gives
\begin{equation}
e^{\frac{2}{n}D(\mu \Vert \gamma)} \leq \left( \cc (\wasser(\mu, \gamma))^{2} + \frac{2}{n} \s (\wasser(\mu, \gamma))^{2} \right) \left(1 + \frac{1}{2n} I(\mu \Vert \gamma) \right).
\end{equation}
 Now the identity $\cc^{2} + \frac{2}{n} \s^{2} = 1$ allows us to conclude the second inequality in the assertion. The first inequality, which compares our result with a Logarithmic Sobolev inequality in a more standard form, follows from the inequality $e^x -1 \ge x$.  
\end{proof}
\begin{rem}
    The Gaussian measure satisfies a Logarithmic Sobolev inequality with constant $1$, that is, $D(\mu \Vert \gamma) \leq \frac{C}{2} I(\mu \Vert \gamma)$, $C=1$, without symmetry or concavity assumptions on $\mu$. 
    Moreover, Bobkov, Gozlan, Roberto and  Samson's result from \cite{BobkovGozlanRobertoSamson14} implies that if $\mu$ is strongly log-concave and centered, we already have 
    $ n \left( e^{\frac{2}{n}D(\mu \Vert \gamma)} - 1 \right) \leq I( \mu \Vert \gamma) $ -- See Corollary  2.2 there and note that indeed $b\le 1$ in this setting by Poincar\'e. Corollary \ref{cor: logsobolev} says that, for even strongly log-concave measures, both of these results improve by a factor of $2$. It is well known that the Poincar\'e constant of the Gaussian measure improves to $1/2$ for even functions. Given the relationship between Logarithmic Sobolev inequalities (stronger) and Poincar\'e inequalities (weaker), one may wonder if Corollary \ref{cor: logsobolev} can be strengthened to hold for all even measures. However, this is not the case; the improvement by a factor of $2$ in the Gaussian Poincar\'e inequality for even functions cannot be lifted to the same improvement in the Logarithmic Sobolev inequality for even measures without further assumptions, as the following example shows. 
    \end{rem}
    \begin{exmpl} \label{exmpl: betterlsineedsslc}
    Let $\mu$ denote the distribution of $tZ$, where $Z$ has the distribution of standard Gaussian in $\R$. Then, all quantities involved can be explicitly computed,
    \begin{equation}
    D(\mu \Vert \gamma) = - \log t + \frac{t^2}{2} - \frac{1}{2}, \textnormal{ and }
    I(\mu \Vert \gamma) = \frac{(t^2 - 1)^2}{t^2}.
    \end{equation}
As $t \mapsto \infty$, we have $\frac{4 D(\mu \Vert \gamma)}{t^{2}} \to 2$ but $\frac{I(\mu \Vert \gamma)}{t^{2}} \to 1$. On the other hand, taking the limit $t \to 1$ shows that our inequality is sharp. 
    \end{exmpl}
The improvement in the Logarithmic Sobolev inequality for even strongly log-concave measures also reflects in the convergence rate to equilibrium. 
\begin{cor}
    Let $P^{\ast}_{t} \mu$ denote the Ornstein-Uhlenbeck evolution of an even strongly log-concave measure $\mu = \mu_{0}$, that is, $\d P^{\ast}_{t} \mu = P_{t} \left( \frac{\d \mu}{\d \gamma} \right) \d \gamma$. Here $P_{t}$ denotes the Ornstein-Uhlenbeck semigroup. Then,
    \begin{equation}
    D(P^{\ast}_{t} \mu \Vert \gamma ) \leq e^{-4t} D(\mu_{0} \Vert \gamma).
    \end{equation}
\end{cor}
\begin{proof}[Proof sketch]
    The proof is a standard argument using the facts that $\frac{\d}{\d t}   D(P^{\ast}_{t} \mu \Vert \gamma ) = -   I(P^{\ast}_{t} \mu \Vert \gamma )$, and that the Ornstein-Uhlenbeck semigroup $P_{t}$ (which acts on functions) preserves the class of even functions as well as the class of log-concave functions. 
\end{proof}
\begin{rem}
    The previous proof uses only the dimension-free part $4D(\mu \Vert \gamma) \leq I(\mu \Vert \gamma)$ in Corollary \ref{cor: logsobolev}. One can use the dimensional dependence available. However, this really improves the inequality only for short times $t$. Long-term behavior seems to be dominated by the ``curvature'' aspect. The interested reader may also consult \cite{BakryBolleyGentil12}. 
\end{rem}

We conclude this section with another typical application of the HWI inequality. 
\begin{cor}[A Talagrand inequality] \label{cor: talagrand}
    Let $\mu$ be an even strongly log-concave probability measure on $\R^n$. Then,
        \begin{equation} \label{eq: talagrand}
\wasser^{2} (\mu , \gamma) \leq - n \log \cos \left( \sqrt{\frac{2}{n}} \wasser(\mu, \gamma) \right) \leq D(\mu \Vert \gamma).
        \end{equation}
\end{cor}
\begin{proof}
    This follows from plugging in $\gamma = \mu_{0}$ and $\mu = \mu_{1}$ in the HWI Inequality \eqref{eq: hwi*}, which makes $I(\mu_{0})=0$. 
\end{proof}
\begin{rem}
    The Gaussian measure satisfies $\wasser^{2} (\mu, \gamma) \leq 2 D(\mu \Vert \gamma)$ for all probability measures $\mu$. Since the Gaussian Talagrand inequality is also stronger than the Poincar\'e inequality (\cite[Section 7]{OttoVillani00}), one may again consider the possibility that Corollary \ref{cor: talagrand} holds for all even measures. However, again, this is not the case. The example presented in the case of the Logarithmic Sobolev inequality serves us here as well. 
    \end{rem}
    \begin{exmpl} \label{exmpl: bettertalagrandneedsslc}
   Let $\mu$ denote the distribution of $tZ$, where $Z$ has the distribution of standard Gaussian in $\R$. Then,
    \begin{equation}
     \wasser^{2} (\mu , \gamma) = (t-1)^2, \textnormal{ and }D(\mu \Vert \gamma) = - \log t + \frac{t^2}{2} - \frac{1}{2}.
\end{equation}
As $t \to \infty$, we have $\frac{\wasser^{2}(\mu, \gamma)}{t^2} \to 1$ but $\frac{D(\mu \Vert \gamma)}{t^{2}} \to \frac{1}{2}$. On the other hand, as before, taking $t \to 1$ shows that our inequality is sharp. 
\end{exmpl}

\bibliographystyle{amsplain}
\bibliography{shoulderofgiants,morereferences}
\end{document}